\documentclass[11pt,a4paper,dvipsnames]{article}

\usepackage[utf8]{inputenc}
\usepackage[english]{babel}
\usepackage{amsmath}
\usepackage{amsthm}
\usepackage{amssymb}
\usepackage{geometry}
\usepackage{a4wide}
\usepackage{bm}
\usepackage{microtype}
\usepackage{mathtools}
\usepackage{csquotes}
\usepackage{authblk}

\theoremstyle{definition}
\newtheorem{theorem}{Theorem}[section]

\newtheorem{lemma}[theorem]{Lemma}
\newtheorem*{lemma*}{Lemma}
\newtheorem*{conjecture*}{Conjecture}
\newtheorem*{lemma''*}{``Lemma''}
\newtheorem{claim}[theorem]{Claim}
\newtheorem*{claim*}{Claim}
\newtheorem{corollary}[theorem]{Corollary}
\newtheorem{question}[theorem]{Question}
\newtheorem{conjecture}[theorem]{Conjecture}

\newcommand		\sign		{\operatorname{sign}}
\newcommand		\mc				\mathcal

\newcommand \M {\cal M}
\newcommand \R  {\mathbb R}
\newcommand \supp {\hbox{supp }}

\title{\bf Minimally rigid tensegrity frameworks}

\author[1]{Adam D. W. Clay}
\author[2,3]{Tibor Jordán}
\author[3]{Sára Hanna Tóth}
\affil[1]{\footnotesize Department of Mathematics, Purdue University, West Lafayette, IN 47907, USA}
\affil[2]{\footnotesize HUN-REN--ELTE Egerváry Research Group on Combinatorial Optimization, Pázmány~Péter~sétány~1/C, Budapest, 1117, Hungary}
\affil[3]{\footnotesize Department of Operations Research, ELTE Eötvös Loránd University, Pázmány~Péter sétány~1/C, Budapest, 1117, Hungary}
\affil[ ]{\footnotesize \textit{E-mail addresses:} {\tt adwclay@gmail.com}, {\tt tibor.jordan@ttk.elte.hu}, {\tt toth.sara.hanna@gmail.com}}

\date{October 6, 2024}

\begin{document}

\maketitle

\begin{abstract}
A $d$-dimensional tensegrity framework $(T,p)$ is an edge-labeled geometric graph in $\R^d$, which consists of 
a graph $T=(V,B\cup C\cup S)$ and a map $p:V\to \R^d$. 
The labels determine whether an edge $uv$ of $T$ corresponds to a fixed length bar in
$(T,p)$, or a cable
which cannot increase in length, or a strut which cannot decrease in length.

We consider minimally infinitesimally rigid $d$-dimensional tensegrity frameworks and 
provide tight upper bounds
on the number of its edges, in terms of the number of vertices and the dimension $d$. We obtain stronger upper bounds in the case when there are no bars and
the framework is in generic position. The proofs use methods from convex geometry and matroid theory. A special case of our results confirms a conjecture of Whiteley from 1987.
We also give an affirmative answer to a conjecture concerning the number of edges of a graph whose three-dimensional rigidity matroid is minimally connected.
\end{abstract}

\section{Introduction}

A {\it tensegrity graph} $T=(V, B\cup C \cup S)$ is an edge-labeled graph
%\footnote{We assume that the graph is simple, unless stated otherwise.} 
on vertex set $V$. 
%According to the labels, each edge $e$ is labelled as a %{\it bar}, a {\it cable}, or a {\it strut}.
We assume that the graph is simple, unless stated otherwise.
According to its label, each edge of $T$ is called a 
{\it bar}, a {\it cable}, or a {\it strut}.
%
%sets, $B$, $C$, and $S$. 
The bars, cables and struts together are the {\it members} of $T$. 
A $d$-dimensional {\it tensegrity framework} $(T,p)$ is a pair, where $T=(V,B\cup C\cup S)$ is a tensegrity graph and $p:V\to \R^d$ is a map, satisfying
$p(u)\not= p(v)$ for each $uv\in B\cup C\cup S$. It is also called a {\it realization} of $T$ in $\R^d$.
The length of a 
member $uv$ of $T$ in $(T,p)$ is the distance between $p(u)$ and $p(v)$.
In a (deformation of a) tensegrity framework the length of each bar is fixed, the
cables cannot increase in length, and the struts cannot decrease in length. 
%Hence the length of each bar is fixed.
If each member of $T$ is a bar, we say that $(T,p)$ is a {\it bar-and-joint framework}.
In this case we often denote the tensegrity graph by $G=(V,E)$.
If there are no bars in $(T,p)$, then we call it a {\it cable-strut framework}.

Roughly speaking, a $d$-dimensional tensegrity framework is said to be rigid in $\R^d$
if there is no non-trivial continuous deformation 
%(resp. no deformation at all) 
of $(T,p)$ that respects 
the length constraints of the edges.
We shall consider a stronger and more tractable form of rigidity, called
infinitesimal rigidity.
(Precise definitions are given in the next section.)  
%A framework $(T,p)$ is said to be {\it injective} if $p(u)\not= p(v)$ for each pair $u,v\in V$ of distinct vertices. 
Our main target is to obtain tight upper bounds on the number of members of
minimally infinitesimally rigid tensegrity frameworks in $\R^d$, 
%for all $d\geq 1$.
in terms of the number of vertices and $d$.
Minimality refers to the property that for each member $e$ of $T$ the framework
$(T-e,p)$ is no longer infinitesimally rigid.

It is well-known that every minimally infinitesimally rigid bar-and-joint framework on vertex set $V$, with $|V|\geq d+2$, has exactly 
\begin{equation}
\label{up}
d|V| - \binom{d+1}{2}
\end{equation}
bars.
In the case of minimally infinitesimally rigid tensegrity frameworks this number is not uniquely determined by $|V|$ and $d$. It is not hard to see that the number of members is at least 
%In the case of tensegrity frameworks this number is not uniquely determined by $|V|$ and %$d$. It is not hard to see that the number of members in $(T,p)$ is at least 
$d|V| - \binom{d+1}{2}+1$, but it can also be as large as 
\begin{equation}
\label{up1}  
2d|V| - 2\binom{d+1}{2}
\end{equation}
%$$2d|V| - 2\binom{d+1}{2}$$.
%$(d+1)|V|-\binom{d+2}{2}$, 
See\footnote{In our figures solid lines, dashed lines, and double lines represent
bars, cables, and struts, respectively.} Figure \ref{4v6_c-s}.
Finding a tight upper bound has been an open problem for all $d\geq 1$.
In the first part of the paper we employ methods from convex geometry and show that
(\ref{up1}) is the best possible upper bound for all $d\geq 1$ and $|V|\geq d+2$ (Theorem \ref{th_parallel_bars}). In order to obtain our bounds, and some refinements,
we use and extend a classic result of Steinitz from convex geometry.
We can also characterize the extremal cases if there are no bars.

\begin{figure}[!h]
\begin{center}
\includegraphics[width=4cm]{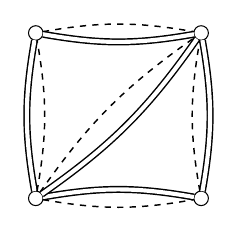}
\caption{\label{4v6_c-s} A minimally rigid cable-strut framework in $\mathbb R^2$ with $4|V|-6$ members.}
\end{center}
\label{1dim}
\end{figure}

In the second part of the paper we consider minimally infinitesimally rigid generic cable-strut
frameworks and show the stronger, tight upper bound
\begin{equation}
\label{up2}  
(d+1)|V| - \binom{d+2}{2}
\end{equation}
for $1\leq d\leq 3$ (Theorem \ref{minrigbound}). 
This result confirms a conjecture of W. Whiteley from 1987 \cite{Whiteley}.
We conjecture that (\ref{up2}) is the
best possible upper bound for all $d\geq 1$ (Conjecture \ref{conj_cs}). We can verify the
bound  $\frac{3}{2}d|V|$, which is roughly half-way between (\ref{up1}) and 
(\ref{up2}), for all $d\geq 1$ (Theorem \ref{thm: edge bounds}).

%We say that $(T,p)$ is 
%{\it generic} if the $d|V|$ coordinates of the points $p(v)$, $v\in V$, do not satisfy any %non-zero polynomial with rational coefficients.

For every fixed $d$, the existence of a single generic infinitesimally rigid $d$-dimensional realization of a 
graph $G$ as a bar-and-joint framework implies that every generic $d$-dimensional
realization of $G$ is infinitesimally rigid in $\R^d$. 
In the case of tensegrity frameworks the situation is different: infinitesimal rigidity
of a generic framework $(T,p)$ depends on $T$ as well as on $p$.
For example, a triangle of two cables and a strut has infinitesimally rigid as well
as non-rigid generic realizations on the line.
See also Figure \ref{notgen} for a two-dimensional example.
This fact makes some of the combinatorial questions concerning
tensegrities substantially more difficult than the corresponding questions for bar-and-joint
frameworks. For example, for each $d\geq 2$ it is an open problem to characterize the 
tensegrity graphs which can be realized as generic rigid tensegrity frameworks in
%weakly rigid or weakly globally
%rigid tensegrity graphs in 
${\mathbb R}^d$. Moreover, testing whether {\em every} generic realization of a
tensegrity graph $T$ is rigid 
%(in which case $T$ is called {\em strongly rigid}) 
is NP-hard, even in ${\mathbb R}^1$ \cite{JJK}.
It also follows that our upper bounds on the number of members in a generic
minimally infinitesimally rigid framework
are not purely combinatorial
properties of the underlying tensegrity graph.

In order to deal with this phenomenon, we use methods from matroid theory.
In particular,
we introduce the circuit decompositions of matroids, which generalize ear-decompositions.
We show that every generic infinitesimally
rigid tensegrity framework $(T,p)$ has a ``geometric''  circuit decomposition
and use it to deduce our upper bounds on the number of members.

\begin{figure}[!h]
\begin{center}
\includegraphics[]{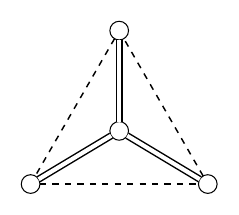}
\includegraphics[]{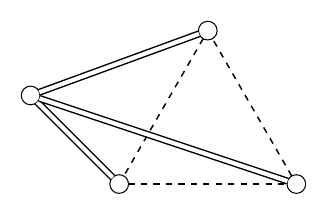}
\caption{\label{notgen} A rigid and a non-rigid generic realization of the same tensegrity graph in the plane.}
\end{center}
\end{figure}

The structure of the paper is as follows. 
In Section \ref{sec:2} we collect the basic definitions and results from
rigidity theory we shall need. 
Section \ref{sec:3} contains the required results of convex geometry, including some
preliminary lemmas as well as extensions of the theorem of Steinitz.
In Section \ref{sec:4} we prove the tight upper bounds on the number of members of
minimally infinitesimally rigid tensegrity frameworks. In Sections \ref{sec:5} and \ref{sec:6}
we introduce circuit decompositions and prove a number of related results concerning
graphs and frameworks, respectively.
Section \ref{sec:7} contains the upper bounds in the case of generic cable-strut frameworks.
By using our proof methods we can also verify a conjecture on minimally
connected rigidity matroids in Section \ref{sec:8}. The paper ends with some concluding
remarks and open problems.

\section{Rigid frameworks and the rigidity matroid}
\label{sec:2}

Let $(T,p)$
 and  $(T,q)$
 be two 
$d$-dimensional realizations of the tensegrity graph $T=(V,B\cup C\cup S)$. 
We say that $(T,p)$
{\it dominates} $(T,q)$  
 if we have 
 $||p(u)-p(v)||=||q(u)-q(v)||$ for all $uv\in B$,
  $||p(u)-p(v)||\geq ||q(u)-q(v)||$ for all $uv\in C$, and
$||p(u)-p(v)||\leq ||q(u)-q(v)||$ for all $uv\in S$. 
In this case we also say that $(T,q)$
{\it satisfies the member constraints} of $(T,p)$.
%, or, if every member of 
% is a bar, that 
% and 
% are equivalent. 
Here $||.||$
 denotes the Euclidean norm in $\R^d$.
If we have $||p(u)-p(v)|| = ||q(u)-q(v)||$ for all $u,v\in V$,
then we say that 
%$(T,p)$
% and $(T,q)$
$p$ and $q$
 are {\it congruent}.

We say that a tensegrity $(T,p)$ 
 is {\it rigid} if there is an $\epsilon > 0$  
 such that for any other realization $(T,q)$ 
 with $||p(v)-q(v)||\leq \epsilon$
 for all $v\in V$
 that satisfies the member constraints of $(T,p)$, 
 $p$ is congruent to $q$. 
It can be shown that a tensegrity framework is not rigid if and only if it is {\it flexible}, which means that 
there is a continuous path $x:[0,1]\to \R^{dn}$ with $x(0)=p$ such that $(T,x(t))$
is a framework that satisfies the member constraints of $(T,p)$,
and $p$ and $x(t)$ are not congruent, for all $0< t\leq 1$, see \cite{RW}.
(Here $n=|V(T)|$ and each map $p:V\to {\mathbb R}^d$ is associated with a point in $\R^{dn}$ in the natural way.)

We next define a different form of rigidity, called infinitesimal rigidity.
An {\it infinitesimal motion} of a tensegrity framework $(T,p)$ is an assignment $\mu:V\to \R^d$ such that for each member 
$uv$ we have 
$(p(u)-p(v))\cdot (\mu(u)-\mu(v)) = 0$, if $uv\in B$, 
$(p(u)-p(v))\cdot (\mu(u)-\mu(v)) \leq 0$, if $uv\in C$, and $(p(u)-p(v))\cdot (\mu(u)-\mu(v)) \geq 0$, if $uv\in S$.
We say that $(T,p)$ is {\it infinitesimally rigid} if every infinitesimal motion
$\mu$ of $(T,p)$ is an infinitesimal isometry of $\R^d$,
that is, $(p(u)-p(v))\cdot (\mu(u)-\mu(v))=0$ holds for all $u,v\in V$.
It is known that infinitesimal rigidity implies rigidity for tensegrity frameworks, see \cite[Theorem 5.7]{RW}.
Infinitesimal rigidity can be characterized by the 
the {\it rigidity matrix} $R(T,p)$ of the framework, which is the coefficient matrix of the above inequalities.  %$(T,p)$ is
This matrix is 
a $|B\cup C\cup S|\times d|V|$ matrix, in which in the row indexed by an edge $uv$ we have
$p(u)-p(v)$ (resp. $p(v)-p(u)$) in the $d$ columns corresponding to $u$ (resp. $v$), and the remaining entries
are zeros.
In the special case, when $(G,p)$ is a bar-and-joint framework, it is well-known that infinitesimal rigidity can be characterized by the rank of $R(G,p)$. Let $G=(V,E)$ and suppose that $|V|\geq d$. 
Then $(G,p)$ is infinitesimally rigid in $\R^d$ if and only if the rank of $R(G,p)$ is equal to $d|V|-{d+1\choose 2}$.
It is easy to see that the graphs of infinitesimally rigid frameworks on less than $d$ vertices must be complete.

\begin{figure}[!h]
\begin{center}
\includegraphics[width=5cm]{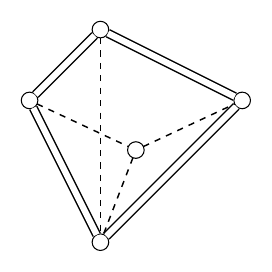}
\includegraphics[width=5cm]{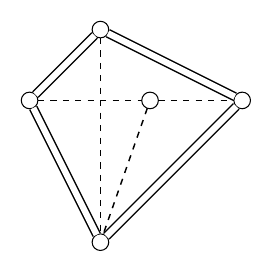}
\caption{\label{rigvsinfrig} Two rigid realizations of the same tensegrity graph in the plane. The first one is infinitesimally rigid, the second one is not.}
\end{center}
\end{figure}

Consider a tensegrity framework $(T,p)$ with $T=(V,B\cup C\cup S)$. Let $\omega: (B\cup C\cup S)\to \R$ be a function which assigns a scalar to each member of $T$. We say that $\omega$ is a {\it dependence}
if it satisfies the following equilibrium condition: 
%A {\it stress} of a tensegrity framework $(T,p)$ of $T=(V,B\cup C\cup S)$ is a function %$\omega: (B\cup C\cup S)\to \R$, which assigns a scalar to each member of $T$ such that
\begin{equation}
\label{eqs}
\sum_{uv\in B\cup C\cup S} \omega(uv) (p(u)-p(v))=0 \ \hbox{ for all}\ v\in V,
\end{equation}
If, in addition, $\omega$ also satisfies the following {\it sign constraints}, then
we call $\omega$ a {\it stress} of $(T,p)$:
\begin{equation}
\label{scC}
\omega(e)\leq 0, \ \hbox{for all}\ e\in C,
\end{equation}
and
\begin{equation}
\label{scS}
\omega(e)\geq 0, \ \hbox{for all}\ e\in S. 
\end{equation}
%$$\omega(e)\geq 0, \ \hbox{for all}\ e\in S,$$
%(resp. $\omega(e)\geq 0$) for all $e\in C$ (resp. $e\in S$) 
%and
%$$\sum_{uv\in E} \omega(uv) (p(u)-p(v))=0 \ \hbox{ for all}\ v\in V.$$
%It is proper if $w(e) > 0$ (resp. $w(e) < 0$) whenever $e$ is a cable (resp. strut).

\noindent
%We call (\ref{eqs}) the equilibrium condition, while (\ref{scC}) and (\ref{scS}) are
%called the
%We call these inequalities the 
%{\it sign constraints}.
Note that there are no sign constraints on the bars.
Thus a dependence of a bar-and-joint framework $(G,p)$ is a stress - it is simply a linear
dependence among the rows of $R(G,p)$.
A {\it proper stress}\footnote{There is a different terminology used in the literature, where  the names 
dependence, stress, and proper stress are replaced by 
stress, proper stress, and strict proper stress, respectively.
We follow the terminology of \cite{RW}.} of $(T,p)$ is a stress of $(T,p)$ satisfying
$\omega(e)< 0$ (resp. $\omega(e)> 0$) for all $e\in C$ (resp. $e\in S$).

The following fundamental lemma is due to Roth and Whiteley.
For a tensegrity framework $T=(V,B\cup C\cup S)$ we use $\overline T$ to denote its underlying graph: it is a graph on vertex set $V$ with edge set $E=B\cup C\cup S$.

\begin{theorem} \cite[Theorem 5.2]{RW}
\label{fund}
    Let $(T, p)$ be a tensegrity framework in $\R^d$. Then $(T, p)$ is infinitesimally rigid if and only if $(\overline T, p)$ is infinitesimally rigid and there exists a proper stress of $(T, p)$.
    %\label{thm: infinitesimally rigid equivalences}
\end{theorem}

For a set $A$ of members of $T$ we use $R_A(T,p)$ to denote the
submatrix of $R(T,p)$ induced by the rows of $A$.

The {\it support} $\supp \omega$ of a stress $\omega$ of $(T,p)$ is the set $\{e\in B\cup C\cup S: \omega(e)\not= 0\}$ of members with non-zero stress.
%Let $R(T,p)$ be the rigidity matrix of $(T,p)$. For a set $A$ of members of $T$ we use %$R_A(T,p)$ to denote the
%submatrix of $R(T,p)$ induced by the rows of $A$.
We shall need one more lemma from \cite{RW}.

\begin{lemma} \cite[Lemma 5.5]{RW}
\label{own}
Let $(T,p)$ be a  tensegrity framework in $\R^d$ with
$T = (V, B\cup C \cup S)$, 
%in $\R^d$ 
and let $e \in B\cup C \cup S$. If there exists a stress of $(T, p)$ with $e$ in its support, then there exists a stress $\omega$ of $(T, p)$ with $e$ in its support and such that {\em rank}$\ R_A(T,p) = |A|$ for every $A \subsetneq \supp \omega$.
    \label{lmm: minimally dependent stress support}
\end{lemma}

In $\R^1$ the rigidity matrix of $(T,p)$ can be obtained from the incidence matrix of $\overline T$ by multiplying each row by a non-zero real number.
%and then
%doubling some of
%the rows.
%, and multiplying each row by a non-zero real number. 
Hence the edges in the support of the proper stress $\omega$ given by
Lemma \ref{own} form a cycle in $T$.

An infinitesimally rigid tensegrity framework $(T,p)$ in $\R^d$ is called {\it minimally infinitesimally rigid} in $\R^d$ if $(T-e,p)$ is not 
infinitesimally rigid in $\R^d$ for every member $e$ of $T$.
We say that a tensegrity framework $(T,p)$ is {\it minimally properly stressed}
if $(T,p)$ has a proper stress, but $(T-e,p)$ has no proper stress for all members $e$ of $T$.

By applying Theorem \ref{fund} we obtain:

\begin{lemma}
\label{minps}
Let  $(T,p)$ be an infinitesimally rigid tensegrity framework in ${\R}^d$.
Then $(T,p)$ is minimally infinitesimally rigid if and only if
$(T,p)$ is minimally properly stressed.
\end{lemma}

\begin{proof}
The ``if'' direction is immediate from Theorem \ref{fund}. To see the only if direction suppose that $(T,p)$ is minimally infinitesimally rigid, but
$(T-e,p)$ is properly stressed for some member $e$ of $T$. Since $(\overline T,p)$  is infinitesimally rigid and
has a proper stress, $(\overline{T-e},p)$ is infinitesimally rigid. 
Thus $(T-e,p)$ is also infinitesimally rigid by Theorem \ref{fund}.
%Then $\bar T$ is
%rigid
\end{proof}

%For a set $A$ of members of $T$ we use $R_A(T,p)$ to denote the
%submatrix of $R(T,p)$ induced by the rows of $A$.

We say that $(T,p)$ is {\it generic}\footnote{In \cite{RW} the authors
use the term general position to mean generic. It should also be noted that in rigidity theory the term generic sometimes means that the set of coordinates of the points are algebraically independent over the rationals. In some of our statements we use the above, substantially milder assumption.}
 if 
 $$rank R_A(T,p) = \max \{ rank R_A(T,q) : (T,q)\ \hbox{is a}\ d\hbox{-dimensional realization of}\ T\}$$
 for every nonempty $A\subseteq B\cup C\cup S$.
 %In particular, every generic framework is in general position. 
%  Thus a stress of a bar-and-joint the framework $(G,p)$ is just a linear
% dependency among the rows of $R(G,p)$.
It is known that a generic tensegrity framework is 
infinitesimally rigid if and only if it is rigid, see \cite[Theorem 5.8]{RW}.

For more details on rigid bar-and-joint and tensegrity frameworks see \cite{CN,SW}.

\subsection{The rigidity matrix and the rigidity matroid}

The rigidity matroid of a graph $G$ is a matroid defined on the edge set
of $G$ which reflects the rigidity properties of all generic realizations of
$G$. 
%For a general introduction to matroid theory we refer the reader to \cite{oxley}.
Let $(G,p)$ be a (bar-and-joint) realization of a graph $G=(V,E)$ in ${\mathbb R}^d$.
%Recall that the {\it rigidity matrix} of the framework %$(G,p)$
%is the matrix $R(G,p)$ of size
%$|E|\times d|V|$, where, for each edge $v_iv_j\in E$, in %the row
%corresponding to $v_iv_j$,
%the entries in the $d$ columns corresponding to vertices %$v_i$ and $v_j$ contain
%the $d$ coordinates of
%$(p(v_i)-p(v_j))$ and $(p(v_j)-p(v_i))$, respectively,
%and the remaining entries
%are zeros.
%See \cite{Whlong} for more details.
The rigidity matrix $R(G,p)$ defines
the {\it rigidity matroid}  of $(G,p)$ on the ground set $E$.
In this matroid a set $A\subseteq E$ is independent if and only if
the rows of $R_A(G,p)$ are linearly independent. It is not hard to see that any pair 
$(G,p)$ and $(G,q)$ of generic realizations of $G$ have the same rigidity matroid.
We call this the $d$-dimensional {\it rigidity matroid}
${\cal R}_d(G)=(E,r_d)$ of the graph $G$.
Thus if $(G,p)$ is generic, then the
rank of $R_A(G,p)$ is equal to $r_d(A)$ for all $A\subseteq E$.
We denote the rank of ${\cal R}_d(G)$ by $r_d(G)$.
We say that a graph $G$ is {\it rigid} in ${\mathbb R}^d$ if either $G$ has less than $d$ vertices and $G$ is complete, or $G$ has at least $d$ vertices and 
$r_d(G) = d|V| - \binom{d+1}{2}$ holds. 

A graph $G=(V,E)$ is {\it ${\cal R}_d$-independent} if $r_d(G)=|E|$ and it is an 
{\it ${\cal R}_d$-circuit}  if it is not ${\cal R}_d$-independent, but every proper 
subgraph $G'$ of $G$ is ${\cal R}_d$-independent.

Let ${\cal M}$ be a matroid on ground set $E$ with rank function $r$. 
We can define a relation on the pairs of elements of $E$ by
saying that $e,f\in E$ are
equivalent if $e=f$ or there is a circuit $C$ of ${\cal M}$
with $\{e,f\}\subseteq C$.
This defines an equivalence relation. The equivalence classes are 
the {\it connected components} of ${\cal M}$.
The matroid is said to be {\it connected} if there is only one equivalence class.
%Given a graph $G=(V,E)$, the subgraphs induced by the edge sets of the connected %components
%of ${\cal R}_d(G)$ are the
%{\it ${\cal R}_d$-connected components} of $G$.
The graph is said to be {\it ${\cal R}_d$-connected} if ${\cal R}_d(G)$ is connected.
For a general introduction to matroid theory see \cite{Oxley}.

\section{A theorem of Steinitz and its refinements}
\label{sec:3}

The proofs of our upper bounds on the number of members of a minimally infinitesimally rigid tensegrity framework
employ (extensions of) a classic result of convex geometry due to Steinitz~\cite{Steinitz} (Theorem \ref{steinitz} below), see also~\cite[Theorem 4.22.]{Soltan}.
In this section we state this result and prove some extensions and refinements.
We start with 
some basic definitions and lemmas.

\subsection{Preliminary lemmas}

Let $X =\{x_1, \ldots , x_m\} \subset \mathbb{R}^n$ be a finite set of points. 
The \textit{(linear) span} of $X$ is the subspace $\text{span}(X)=\{ \sum_{i=1}^m \lambda_i x_i : \lambda_1, \ldots, \lambda_m \in \R\}$.
The \textit{convex hull} of $X$ is the set $\text{conv}(X)=\{ \sum_{i=1}^m \lambda_i x_i : \sum_{i=1}^m \lambda_i = 1,~ 0 \leq \lambda_1, \ldots, \lambda_m \in \R\}$. 
The polytope conv($X$) is \textit{$k$-dimensional}, if there exists a $k$-dimensional affine subspace in $\R ^n$ containing conv($X$) and there is no $(k-1)$-dimensional affine subspace containing conv($X$).
%Equivalently: conv($X$) contains a $k$-dimensional ball and no $(k+1)$-dimensional ball is contained in conv($X$). 
A point $x$ is in the \textit{relative interior} of conv($X$) if conv($X$) is $k$-dimensional and there exists a $k$-dimensional ball centered at $x$ contained in conv($X$).
For completeness we include the proofs of the next two preliminary lemmas.

\begin{lemma} \label{lem:relint}
    Let $X \subset \mathbb{R}^n$ be a finite set of points.
    A point $x\in \mathbb{R}^n$ is in the relative interior of conv$(X)$ if and only if there is a strictly positive convex combination of the elements of $X$ resulting in $x$.
\end{lemma}

\begin{proof}
    Let $X=\{x_1, \ldots, x_m\}$, and suppose that conv($X$) is $k$-dimensional. First, assume that $x$ is the origin.
    If the origin is in the relative interior of conv($X$) then 
    $-\varepsilon \cdot \sum_{i=1}^m x_i$ is in $\text{conv}(X)$ for a sufficiently small positive $\varepsilon$, thus
    $$
    - \varepsilon \sum_{i=1}^m x_i = \sum_{i=1}^m \lambda_i x_i,
    $$
    where $\lambda_1, \ldots, \lambda_m \geq 0$ and $\sum_{i=1}^m \lambda_i =1$. Then
    $$
    0 = \sum_{i=1}^m \frac{\lambda_i + \varepsilon}{1+m \cdot \varepsilon} x_i
    $$
    is a strictly positive convex combination of the elements of $X$ resulting in zero.

    Conversely, let $\sum_{i=1}^m \alpha_i x_i =0$ be a strictly positive convex combination of the elements of $X$ resulting in the origin. It is enough to show that for any $h \in \text{span}(X)$ there exists a positive $\varepsilon$ such that $\varepsilon h$ is in conv($X$). Let $h \in \text{span}(X)$ and $h= \sum_{i=1}^m \lambda_i x_i$. Denote $\lambda = \sum_{i=1}^m \lambda_i$. 
    If $\lambda \neq 0$, then there exists $\mu \in \R$ with the same sign as $\lambda$ and $|\mu|$ being sufficiently small such that
    $$
    \mu \sum_{i=1}^m \frac{\lambda_i}{\lambda} x_i + (1-\mu) \sum_{i=1}^m \alpha_i x_i = \frac{\mu}{\lambda} h
    $$
    is a convex combination of the elements of $X$ resulting in $\varepsilon h$, where $\varepsilon=\frac{\mu}{\lambda} >0$.  
    If $\lambda =0$ then for a sufficiently small positive $\varepsilon$,
    $$
    \varepsilon \sum_{i=1}^m \lambda_i x_i + \sum_{i=1}^m \alpha_i x_i = \varepsilon h
    $$
    is a convex combination of the elements of $X$ resulting in $\varepsilon h$.

    Finally, observe that if $x$ is not the origin then $x$ is in the relative interior of conv($X$) if and only if the origin is in the relative interior of conv($X'$), where $X'=\{x_1-x, \ldots, x_m-x\}$. Also, there is a strictly positive convex combination of the elements of $X$ resulting in $x$ if and only if there is a strictly positive convex combination of the elements of $X'$ resulting in the origin, since if $\sum_{i=1}^m \alpha_i =1$ then
    $$
    0=\sum_{i=1}^m \alpha_i (x_i-x)
    $$
    is equivalent to
    $$
    x=\sum_{i=1}^m \alpha_i (x_i-x) +x = \sum_{i=1}^m \alpha_i x_i - x \sum_{i=1}^m \alpha_i +x = \sum_{i=1}^m \alpha_i x_i.
    $$
\end{proof}

\begin{lemma} \label{conv_lemma}
    Let $X \subset \mathbb{R}^n$ be a finite set of points. If there exists a strictly positive combination of the elements of $X$ resulting in zero then
    \begin{itemize}
        \item the origin is in the relative interior of conv($X$) and
        \item the dimension of the subspace span($X$) is equal to the dimension of the polytope conv($X$).        
    \end{itemize}
\end{lemma}

\begin{proof}
    Let $X=\{x_1, \ldots, x_m\}$ and $\sum_{i=1}^m \lambda_i x_i$ be a strictly positive combination of the elements of $X$ resulting in zero. Denote $\lambda = \sum_{i=1}^m \lambda_i$. Then $\sum_{i=1}^m \frac{\lambda_i}{\lambda} x_i$ is a strictly positive convex combination resulting in zero, thus the origin is in the relative interior of conv($X$) by Lemma~\ref{lem:relint}.

    Assume that span($X$) is $k$-dimensional.
    Since conv($X$) is contained in span($X$), conv($X$) is at most $k$-dimensional.
    Assume that conv($X$) is contained in a $(k-1)$-dimensional affine subspace $H \subset \text{span}(X)$. The origin is in conv($X$), thus it is in $H$. Therefore, $H$ is a subspace. Since $H$ contains all elements of $X$, it contains span($X$), contradicting the fact that $H$ is $(k-1)$-dimensional. 
    \end{proof}

%\subsection{Steinitz and extensions}

The following classic result is due to Steinitz (see also
\cite{Soltan}).

\begin{theorem}[Steinitz \cite{Steinitz}] \label{steinitz}
Let $X \subset \mathbb{R}^n$ be a finite set of points and let $x\in \mathbb{R}^n$. Suppose that conv$(X)$ is $k$-dimensional and $x$ is in the relative interior of conv$(X)$. Then\\
(i) there is a subset $Y \subseteq X$ of at most $2k$ points such that conv($Y$) is $k$-dimensional and $x$ is in the relative interior of conv$(Y)$.\\
(ii) Furthermore, if $Y$ is a minimal subset of $X$ such that conv($Y$) is $k$-dimensional and $x$ is in the relative interior of conv($Y$), then either $|Y| \leq 2k - 1$, or $Y$ consists of $2k$ points
$\{y_1,y_2,\dots,y_{2k}\}$ such that 
the triple $y_{2i-1},x,y_{2i}$ is collinear for all 
$1\leq i\leq k$.
%collinear in pairs with $x$ \cite{Soltan}.
\end{theorem}

We shall use Theorem \ref{steinitz} as well as some of its refinements that we
prove in the rest of this section.
For a subspace $H \subseteq \R^n$ the \textit{orthogonal complement} of $H$ is the subspace containing the elements of $\R^n$ orthogonal to each element of $H$.

\begin{lemma} \label{lem:steinitz_polytope}
    Let $X \subset \mathbb{R}^n$ be a finite set of points and let $x\in \mathbb{R}^n$. Suppose that conv($X$) is $k$-dimensional and $x$ is in the relative interior of conv($X$).
    Let $Z \subseteq X$ such that conv($Z$) is $k'$-dimensional and $x$ is in the relative interior of conv($Z$). Then there is a subset $Y \subseteq X$ of at most $2(k-k')+|Z|$ points such that conv($Y$) is $k$-dimensional, $x$ is in the relative interior of conv($Y$) and $Z \subseteq Y$.
\end{lemma}

\begin{proof}
    We can assume without loss of generality that $x$ is the origin.
    
    Let $U$ denote the subspace span($Z$), which is $k'$-dimensional by Lemma~\ref{conv_lemma}, and let $W$ denote the subspace obtained by the intersection of the orthogonal complement of $U$ and span($X$).
    For a point $h \in \text{span}(X)$, denote the component of $h$ in $U$ by $h^U$, and the component in $W$  by $h^W$ (thus, $h=h^U+h^W$), and for a subset $Y \subseteq X$ denote $Y^W = \{ y^W: y \in Y\}$ the orthogonal projection of $Y$ to $W$.

    By Lemma~\ref{conv_lemma}, span($X$) is $k$-dimensional, span($Z$) is $k'$-dimensional, hence $W$ is $(k-k')$-dimensional.
    Each $h \in W$ can be expressed as a linear combination of elements of $X$, and orthogonally projecting both sides of this linear expression onto $W$, we obtain $h$ as a linear combination of elements in $X^W$ (since the orthogonal projection eliminates the $U$-directional components of the elements of $X$). So $\text{span}(X^W)=W$, and therefore span($X^W$) is $(k-k')$-dimensional.
    Moreover, if we take a strictly positive combination of the elements of $X$ expressing the origin, then projecting this orthogonally onto $W$ yields to a strictly positive combination of the elements of $X^W$ resulting in the origin.
    
    Thus, by Lemma~\ref{conv_lemma}, conv($X^W$) is a $(k - k')$-dimensional polytope with the origin in its relative interior. By applying Theorem~\ref{steinitz} to $X^W$ we obtain that there exists $Y \subseteq X$ with at most $2(k-k')$ elements, such that conv$(Y^W)$ is $(k - k')$-dimensional and it contains the origin in its relative interior.

    Then $Y \cup Z$ has at most $2(k-k')+|Z|$ elements. Now we prove that conv$(Y \cup Z)$ is $k$-dimensional and contains the origin in its relative interior.
    %by showing that the two conditions of Lemma~\ref{conv_lemma} hold.    

    Let $Z = \{ z_1, \ldots, z_{|Z|} \}$ and $Y = \{ y_1, \ldots, y_{|Y|} \}$. First we show that the elements of $Y \cup Z$ generate span($X$). Let $h \in \text{span}(X)$. Since by Lemma~\ref{conv_lemma}, the elements of $Y^W$ generate $W$, we can express $h^W$ as
    $$
    h^W = \sum_{i=1}^{|Y|} \alpha_i y_i^W.
    $$
    Combining the corresponding elements in $Y$ with the same coefficients we obtain
    $$
    \sum_{i=1}^{|Y|} \alpha_i y_i = \sum_{i=1}^{|Y|} \alpha_i (y_i^W + y_i^U) = h^W + \sum_{i=1}^{|Y|} \alpha_i y_i^U.
    $$
    Since $\sum \alpha_i y_i^U \in \text{span}(Z)$, %and the elements of $Z$ generate span($X$)
    $$
    \sum_{i=1}^{|Y|} \alpha_i y_i^U = \sum_{i=1}^{|Z|} \beta_i z_i,
    $$
    and by $h^U \in \text{span}(Z)$
    $$
    h^U = \sum_{i=1}^{|Z|} \gamma_i z_i.
    $$
    So
    $$
    h= h^W + h^U = \sum_{i=1}^{|Y|} \alpha_i y_i - \sum_{i=1}^{|Z|} \beta_i z_i + \sum_{i=1}^{|Z|} \gamma_i z_i.
    $$
    Thus, $h$ can be expressed as a linear combination of elements of $Y \cup Z$.

    Similarly, we can show that there is a strictly positive combination of the elements of $Y \cup Z$ resulting in the origin. We chose the set $Y$ such that 
    %there exists a strictly positive combination of its elements resulting in the origin.
    $$
    \sum_{i=1}^{|Y|} \alpha_i y_i^W = 0,
    $$
    where $\alpha_1, \ldots, \alpha_{|Y|} >0$.
    Combining the corresponding elements in $Y$ with the same coefficients, we obtain 
    $$
    \sum_{i=1}^{|Y|} \alpha_i y_i = \sum_{i=1}^{|Y|} \alpha_i (y_i^W + y_i^U) = 0 + \sum_{i=1}^{|Y|} \alpha_i y_i^U \in \text{span}(Z).
    $$
    Since the origin is in the relative interior of conv($Z$), for a sufficiently small positive $\varepsilon$, the vector $\varepsilon \cdot \left(-\sum \alpha_i y_i^U\right)$ can be expressed as a strictly positive convex combination of elements in $Z$, thus
    $$
    - \sum_{i=1}^{|Y|} \alpha_i y_i^U = \sum_{i=1}^{|Z|} \beta_i z_i,
    $$
    where  $\beta_1, \ldots, \beta_{|Z|} > 0$.
    Therefore, we have
    $$
    \sum_{i=1}^{|Y|} \alpha_i y_i + \sum_{i=1}^{|Z|} \beta_i z_i = 0,
    $$
    which is a strictly positive combination of elements in $Y \cup Z$ resulting in the origin.

    Therefore, by Lemma~\ref{conv_lemma}, conv$(Y \cup Z)$ is $k$-dimensional and it contains the origin in its relative interior.
\end{proof}

The next corollary is immediate from Lemma~\ref{lem:steinitz_polytope}. It also follows from %the result of Bonnice and Reay proved in~
a result of \cite{BonniceReay}.

\begin{corollary} \label{cor:BR}
    Let $X \subset \mathbb{R}^n$ be a finite set of points and let $x\in \mathbb{R}^n$. Suppose that conv($X$) is $k$-dimensional and $x$ is in the relative interior of conv($X$). Let $k'$ be the dimension of the highest dimensional simplex with vertices in $X$ and having $x$ in its relative interior. Then there is a subset $Y \subseteq X$ of at most $2k-k'+1$ points such that conv($Y$) is $k$-dimensional and $x$ is in the relative interior of conv($Y$).
\end{corollary}

Lemma~\ref{lem:steinitz_polytope} also implies the
following version.
%The following result also follows from Lemma~\ref{lem:steinitz_polytope}.

\begin{corollary} \label{steinitz_with_pairs}
    Let $X \subset \mathbb{R}^n$ be a finite set of points and conv($X$) be a $k$-dimensional polytope with the origin in its relative interior. Let $Z \subseteq X$ such that $Z = Z^{+} \dot\cup Z^{-}$ where $Z^{-} = \{-z: z \in Z^{+} \}$. Then there is a subset $Y \subseteq X$ of at most $2k$ points such that conv($Y$) is $k$-dimensional, the origin is in the relative interior of conv($Y$) and for each point $z \in Z$ either both $z$ and $-z$ are in $Y$ or neither of them is.
    %for each point $z \in Z$ set $Y$ either contains both $z$ and $-z$ or it does not contain either of them
\end{corollary}

\begin{proof}
    Let $Z_1 \subseteq Z^+$ be a maximal linearly independent subset of $Z^+$, and let $Z_2 = \{-z: z \in Z_1 \} \subseteq Z^-$.
    The sum of the elements of $Z_1 \cup Z_2$ is zero, thus, by Lemma~\ref{conv_lemma}, conv($Z_1 \cup Z_2$) is $|Z_1|$-dimensional. Then by Lemma~\ref{lem:steinitz_polytope}, there is a subset $Y \subseteq X$ of at most $2(k-|Z_1|)+2|Z_1|=2k$ points such that conv($Y$) is $k$-dimensional, the origin is in the relative interior of conv($Y$) and $Z_1 \cup Z_2 \subseteq Y$, so for each point $z \in Z$ either both $z$ and $-z$ are in $Y$ or neither of them is.
\end{proof}

\section{The upper bounds}
\label{sec:4}

%\subsubsection{With bars} \label{sec:with_bars}

Recall that a $d$-dimensional infinitesimally rigid tensegrity framework $(T,p)$ with $T=(V,B\cup C\cup S)$ is minimally infinitesimally rigid, if $(T-e,p)$ is not
infinitesimally rigid for all $e\in B\cup C\cup S$.
In this section we prove a tight upper bound on the number of members of a
minimally infinitesimally rigid tensegrity framework in terms of $d$ and $|V|$.

If $(T,p)$ is infinitesimally rigid with $|V|\leq d+1$, then there must be a bar (or a parallel pair of a cable and a strut) between each pair $u,v\in V$, which makes
this special case uninteresting. So we shall always assume that $|V|\geq d+2$ holds.

For a $d$-dimensional tensegrity framework $(T,p)$ let $R^{-}(T,p)$ be the matrix obtained from $R(T,p)$ by multiplying each row  
corresponding to a cable by $-1$,
%in the rigidity matrix $R(T,p)$ 
%with their negatives, 
and adding a new row $-R_e(T,p)$, whenever $R_e(T,p)$ is a
row corresponding to some bar $e$.
%
%the negatives of the rows of $R(T,p)$ corresponding to bars (
So there are $2|B| + |C \cup S|$ rows in $R^{-}(T,p)$. 
Note that the rank of $R^{-}(T,p)$ is equal to the rank of $R(T,p)$.

\begin{theorem} \label{th_parallel_bars}
    Let $(T,p)$ be a minimally infinitesimally rigid 
    tensegrity framework in $\R^d$, where $T=(V,B \cup C \cup S)$ and $|V| \geq d+2$. Then 
    \begin{equation}
    \label{upgen}
    2|B| + |C \cup S| \leq 2 \cdot \left( d|V|-\binom{d+1}{2}\right).
    \end{equation}
\end{theorem}

\begin{proof}
Let $N=d|V|-\binom{d+1}{2}$ and let
%For a $d$-dimensional tensegrity framework $(T,p)$ let $R'(T,p)$ be the matrix obtained by replacing the rows corresponding to cables in the rigidity matrix $R(T,p)$ with their negatives, and adding the negatives of the rows of $R(T,p)$ corresponding to bars (so there are $2|B| + |C \cup S|$ rows of $R'(T,p)$). Note that the rank of $R'(T,p)$ is equal to the rank of $R(T,p)$. 
Let $X$ denote the set of points in $\R ^{d|V|}$ whose coordinates are given by the rows of $R^{-}(T,p)$, and let $Z \subseteq X$ denote the $2|B|$ points corresponding to (pairs of) rows associated with the bars.

Since $(T,p)$ is infinitesimally rigid and $|V|\geq d+2$, 
Theorem \ref{fund} implies that
%According to Theorem~\ref{RW} by Roth and Whiteley, the tensegrity framework $(T,p)$ with $|V| \geq d+2$ is infinitesimally rigid if and only if 

\begin{itemize}
    \item[(1)] $(\overline{T},p)$ is infinitesimally rigid, or equivalently, the rank of $R^{-}(T,p)$ is equal to $N$,
    or equivalently, the subspace span($X$) is $N$-dimensional, and
    \item[(2)] there exists a proper stress of $(T,p)$, or equivalently, there exists $\omega \in \R ^E$ such that $\omega >0$ and $\omega \cdot R^{-}(T,p) = 0$, or equivalently,
    there is a strictly positive combination of the elements of $X$ resulting in zero.
    %or equivalently, the origin is in the relative interior of conv($X$).
\end{itemize}

\noindent Suppose, for 
%Assume for 
a contradiction, that 
%$(T,p)$ is a minimally infinitesimally rigid realization of $T=(V,B \cup C \cup S)$ in $\R^d$ with 
$2|B| + |C \cup S| \geq 2N + 1$ holds.
%and $|V| \geq d+2$.
We can use
%Since $(T,p)$ is infinitesimally rigid, 
(1), (2), and 
%Therefore, by 
Lemma~\ref{conv_lemma} to deduce that conv($X$) is an $N$-dimensional polytope containing the origin in its relative interior.

By Corollary~\ref{steinitz_with_pairs}, there exists a set $Y \subset X$ with $|Y|\leq 2N$ such that conv($Y$) is an $N$-dimensional polytope with the origin in its relative interior, and for each point $z \in Z$ 
we have $|Y\cap \{z,-z\}|\in \{0,2\}$.
%either both $z$ and $-z$ are is $Y$ or neither of them is. 
Thus for each bar $b \in B$ we have that either both of the corresponding points of $X$ are in $Y$ or neither of them is. Therefore, the matrix $R^{-}_Y(T,p)$ 
induces a subframework of $(T,p)$, that we shall denote
by $(T_Y,p)$.
%
%corresponding to the
%points of $Y$ induce the matrix so deleting the edges %corresponding to $X - Y$ results in a subframework of %$(T,p)$.
Moreover, Lemma~\ref{conv_lemma},
the choice of $Y$, and Theorem \ref{fund} imply that
%span($Y$) is $N$-dimensional, thus both (1) and (2) hold %for the submatrix of $R'(T,p)$ only containing the rows corresponding to $Y$. So the tensegrity framework obtained from $(T,p)$ by deleting the edges corresponding to $X - Y$ 
$(T_Y,p)$
is infinitesimally rigid. 
Since the addition of new members to an infinitesimally rigid tensegrity framework preserves infinitesimally rigidity, we obtain that $(T-e,p)$ is 
infinitesimally rigid for any member $e$ whose row (or whose
pair of rows) in $R^{-}(T,p)$ is disjoint from $Y$.
By our assumption, we have
$2|B| + |C \cup S| \geq 2N + 1$, so such an $e$ exists.
%then there exists an edge such that by removing it (1) and %(2) still hold, thus the framework remains infinitesimally %rigid,
This contradicts the minimality of $(T,p)$.
\end{proof}

The upper bound of 
Theorem~\ref{th_parallel_bars} is best possible:
equality holds in (\ref{upgen}) if
$(T,p)$ is obtained from a minimally infinitesimally rigid $d$-dimensional bar-and-joint framework by replacing some of the bars by parallel cable-strut pairs.
There exist different types of extremal examples, too, see Figures
\ref{4v6} and \ref{6v12}. The latter figure is from \cite[Figure 15]{con}.

\begin{figure}[!h]
\begin{center}
\includegraphics[width=5cm]{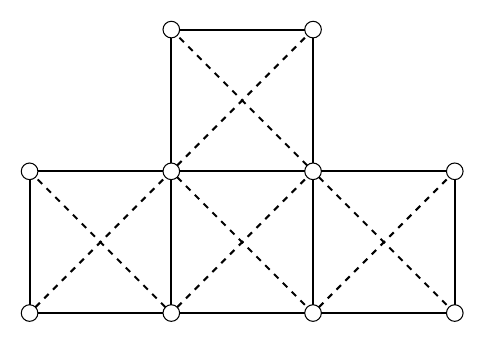}
\caption{\label{4v6} A minimally infinitesimally rigid tensegrity framework in the plane with
$2|B|+|C\cup S|=4|V|-6$.}
\end{center}
\end{figure}

\begin{figure}[!h]
\begin{center}
\includegraphics[width=5cm]{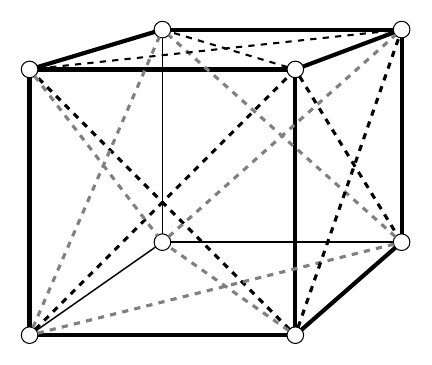}
\caption{\label{6v12}  A minimally infinitesimally rigid tensegrity framework in ${\mathbb R}^3$ with
$2|B|+|C\cup S|=6|V|-12$.}
\end{center}
\end{figure}

\subsection{Frameworks with no bars}

In the rest of this section we consider tensegrity frameworks which contain no bars, but
may contain parallel cable-strut pairs (and hence $T$ may not be simple). 
Although frameworks of this type can achieve equality in (\ref{upgen}),
we can say more about the extremal cases.

%We can achieve equality in (\ref{upgen}) by such a framework, as noted earlier.

\begin{theorem} \label{th_parallel}
    Let $(T,p)$ be a minimally infinitesimally rigid tensegrity framework in $\R^d$ which contains no bars, where $T=(V,C \cup S)$ and $|V| \geq d+2$. Then 
    $$
    |C \cup S| \leq 2 \cdot \left( d|V|-\binom{d+1}{2}\right).
    $$
    Furthermore, equality holds if and only if $(T,p)$ is obtained from a minimally rigid bar-and-joint framework by replacing each bar by a cable and strut pair.
\end{theorem}

\begin{proof}
The first statement follows\footnote{In fact it follows from a simpler, weaker version of Theorem \ref{th_parallel_bars}, whose proof employs Theorem \ref{steinitz}(i) directly, rather than Corollary \ref{steinitz_with_pairs}.} from Theorem \ref{th_parallel_bars} by putting
$B=\emptyset$.
To see the second part observe that
by Theorem \ref{steinitz}(ii)
a minimally infinitesimally rigid tensegrity framework $(T,p)$ has exactly $2\cdot \left( d|V|-\binom{d+1}{2} \right)$ edges if and only if the points in $\R ^{d|V|}$ whose coordinates are the rows of $R^{-}(T,p)$ are collinear in pairs with the origin. 
Such a pair must correspond to a parallel cable-strut pair
in $(T,p)$.
\end{proof}

%Note that we do not use here that $p$ is generic, so we proved a little stornger statement than in Theorem~\ref{thm:2v-3}.

Now we use Corollary \ref{cor:BR} to show that 
if the framework is generic and not extremal, then there is a gap
between its size and the extremal value.

%there is a gap between the
%size of the extremal frameworks of Theorem~\ref{th_parallel}
%and the number of members in frameworks which are not extremal.

%prove a lemma, which is a consequence of Theorem~\ref{steinitz} of Steinitz, and then use it to show a better upper bound for frameworks that are not extremal examples of Theorem~\ref{th_parallel}. This lemma will also be useful to prove a generalization of Theorem~\ref{th_parallel} in Section~\ref{sec:with_bars}.

%This implies the following modification of Theorem~\ref{th_parallel}, which improves the upper bound, if the tensegrity framework is generic and contains a properly stressed $\mathcal{R}_d$-circuit.

%Here, $\mathcal{R}_d(T)$ denotes the rigidity matroid of the simple graph obtained from $T$ by replacing each set of parallel edges in $T$ with a single edge.

\begin{lemma} \label{cor1}
    Let $(T,p)$ be a generic minimally infinitesimally rigid realization of $T=(V,C \cup S)$ in $\R^d$ with $|V| \geq d+2$. Let $\omega$ be a stress of $(T,p)$ such that $\text{supp}(\omega)$ is a circuit of $\mathcal{R}_d(T)$. Let $H=\text{supp}(\omega)$, then 
    $$
    |C \cup S| \leq 2 \cdot \left( d|V|-\binom{d+1}{2}\right)-|H|+2.
    $$
\end{lemma}

\begin{proof}
    Let $R_H^{-}(T,p)$ be the matrix obtained by replacing the rows corresponding to cables in the matrix $R_H(T,p)$ with their negatives, and let $X$ denote the set of points in $\R^{d|V|}$ whose coordinates are the rows of $R_H^{-}(T,p)$. Since $\text{supp}(\omega) = H$, there is a strictly positive combination of the elements of $X$ resulting in zero. Moreover, $H$ is a circuit of $\mathcal{R}_d(T)$, so the dimension of span($X$) is $|H|-1=|X|-1$. Therefore, by Lemma~\ref{conv_lemma}, conv($X$) is $(|X|-1)$-dimensional. Thus it is a simplex, and the origin is in the relative interior of conv($X$). Hence  the proof of Theorem~\ref{th_parallel_bars} 
    combined with Corollary~\ref{cor:BR} gives the improved bound
%    we can improve the upper bound on $|Y|$ by Corollary~\ref{cor:BR} with
$|Y| \leq 2N - (|X|-1) +1 = 2N-|H|+2$, as claimed.
\end{proof}

%We next deduce that if the framework is not extremal,
%(in which case it consists a parallel cable-strut pairs forming a minimally %infinitesimally rigid bar-and-joint framework), 
%then there is a gap
%between the number of its members and the extremal value.

\begin{theorem} \label{cor2}
    Let $(T,p)$ be a generic minimally infinitesimally rigid tensegrity framework 
    in $\R^d$ with no bars. Let  $T=(V,C \cup S)$ and suppose that $|V| \geq d+2$. If there exists a pair of adjacent vertices $u, v$ of $T$ connected by one member only, then
    $$
    |C \cup S| \leq 2 \cdot \left( d|V|-\binom{d+1}{2}\right)-\binom{d+2}{2} +2.
    $$
\end{theorem}

\begin{proof} Let $e=uv$.
    By Lemma \ref{own} and the genericity assumption,
    %Corollary~\ref{cor:RW_circuit},
    there exists a stress $\omega$ of $(T,p)$ for
    which $\supp \omega$ is
    an $\mathcal{R}_d$-circuit and $e\in \supp \omega$.  Since $\mathcal{R}_d$-circuits have at least $\binom{d+2}{2}$ edges, 
    %see \cite{Whiteley},
    the upper bound follows  from  Lemma~\ref{cor1}.
\end{proof}

\section{Circuit decomposition of matroids}
\label{sec:5}

In this section we introduce the combinatorial results and notions we shall
need in the rest of the paper.
We begin with some new observations concerning the circuits of the
rigidity matroid. For a graph $G=(V,E)$ with $|V|\geq d+2$
let $$k_d(G)=d|V|-\binom{d+1}2-r_d(G)$$ denote the ($d$-dimensional) {\it degrees of freedom} of $G$. We have $k_d(G)\geq 0$, with equality if and only if $G$ is rigid in
$\R^d$.

It is well-known that the minimum degree of an ${\cal R}_d$-circuit is at least $d + 1$.
In fact, every ${\cal R}_d$-circuit is $(d+1)$-edge-connected,
%\footnote{In fact
%${\cal R}_3$-circuits are essentially $6$-edge-connected. It is not clear if this fact is %useful.}, 
see \cite{JJdress}.
Moreover, it follows from \cite[Theorem 3.5]{JJsparse} that $K_{d + 2}$ is the only $(d + 1)$-regular ${\cal R}_d$-circuit for $d\geq 2$. 

\begin{lemma}
\label{mindeg}
Let $G = (V,E)$ be a $(d+1)$-regular ${\cal R}_d$-circuit for some $d\geq 1$.
Then either $d=1$ and $G$ is a cycle, or $d\geq 2$ and $G=K_{d + 2}$.
\end{lemma}
%\begin{proof}
%\end{proof}

For $d=1,2$, every 
${\cal R}_d$-circuit $G=(V,E)$ is rigid and has exactly $d|V|-\binom{d+1}2 +1$ edges.
In higher dimensions an ${\cal R}_d$-circuit may not be rigid and its size is not
uniquely determined by $d$ and $|V|$.

\begin{lemma} \label{circuits}
	Let $G = (V,E)$ be an ${\cal R}_d$-circuit for some $d\geq 2$. Then
 \begin{equation}
 \label{ceq1}
%	\begin{enumerate}
 %[label = \textnormal{(\alph{enumi})}, ref = \alph{enumi}] \setlength\parskip{0.05in}
%		 \label{lmm a}
%\item
  |E| = d |V| - \binom{d + 1}2 - k_d (G) + 1,
  \end{equation}
  \begin{equation}
      \label{ceq2}
%  \end{equation}
%
%  \item
%		 \label{lmm b}
		k_d (G) \le \frac12 (d - 1) |V| - \binom{d + 1}2 + 1,\ \hbox{ with equality if and only if}\ G = K_{d + 2},
  \end{equation}
  \begin{equation}
  \label{ceq3}
%
%  \item
%		 \label{lmm c}
		|E| \le (d + 1) |V| - \binom{d + 2}2 - \frac{d + 1}{d - 1} k_d (G),\ \hbox{ with equality if and only if}\ G = K_{d + 2}.
\end{equation}
% \end{enumerate}
\end{lemma}

\begin{proof}
(\ref{ceq1}) follows from the fact that for an ${\cal R}_d$-circuit $G$
we have $r_d(G)=|E|-1$. (\ref{ceq2}) follows from the lower bound on the minimum degree
and Lemma \ref{mindeg}.
%Using the fact that the minimum degree of an $\mc R_d$-circuit is $d + 1$, we can prove part eqref{lmm b} of the %following lemma. Part eqref{lmm c} then 
(\ref{ceq3}) follows by adding $\frac2{d - 1} k_d (G)$ to both sides of (\ref{ceq1}) and
then
applying (\ref{ceq2}).
\end{proof}

Let ${\cal M}$ be a matroid on ground‐set $E$ and let $(C_1, C_2, \ldots, C_t)$ be a sequence of circuits of ${\cal M}$. 
Let $D_0 = \emptyset$, and $D_j = \bigcup_{i = 1}^j C_i$ for $1 \le j \le t$. 
The set $C_i - D_{i - 1}$ is called the {\it lobe} of the circuit $C_i$, and is denoted by $\tilde{C}_i$, for $1 \le i \le t$. 
Following \cite{CH} we say that $(C_1, C_2, \ldots, C_t)$ is a {\it partial ear-decomposition} of ${\cal M}$ if the following properties hold for all $2\leq i\leq t$:
\medskip

(E1')   $C_i - D_{i - 1} \ne \emptyset$ and $C_i \cap D_{i - 1} \ne \emptyset$,

(E2') no circuit $C_i'$ satisfying (E1') has $C_i' - D_{i - 1}$ properly contained in $C_i - D_{i - 1}$.
%\end{enumerate}
\medskip

\noindent
An {\it ear-decomposition} of ${\cal M}$ is a partial ear-decomposition with
$D_t = E$. 
We have:

\begin{lemma} \cite{CH}
Let $\M$ be a matroid. Then
${\cal M}$ has an ear-decomposition if and only if ${\cal M}$ is 
connected. Furthermore, if ${\cal M}$ is connected, then every partial ear-decomposition of ${\cal M}$ can be extended to an ear-decomposition.
\end{lemma}

We next define a new, more general notion.
We say that a sequence  $(C_1, C_2, \ldots, C_t)$ of circuits is a {\it partial circuit decomposition} of ${\cal M}$ if 
%$D_t = E$ and 
the following properties hold for all $2 \le i \le t$:
\medskip
%\begin{enumerate}[label = \textnormal{(E\arabic{enumi})}, ref = E\arabic{enumi}]
%	(E1)   $C_i \cap D_{i - 1} \ne \emptyset$,
	
	(E1)
	$C_i - D_{i - 1} \ne \emptyset$,
	
	(E2) no circuit $C_i'$ satisfying (E1) has $C_i' - D_{i - 1}$ properly contained in $C_i - D_{i - 1}$.
%\end{enumerate}
\medskip

\noindent
A {\it circuit decomposition} of ${\cal M}$ is a partial circuit decomposition with
$D_t = E$. 
%The set $C_i - D_{i - 1}$ is called the {\it lobe} of the circuit $C_i$ in the circuit %decomposition, and is denoted by $\tilde{C}_i$, for $1 \le i \le t$. 
%It's known that $r(\mc D_i) - r(\mc D_{i - 1}) = |\tilde{\mc C}_i| - 1$ for an ear-decomposition. The following lemma states that this remains true even if we drop {E1} as a requirement.

An element $e\in E$ is a {\it bridge} if $r(E-e)=r(E)-1$ holds.
We call ${\cal M}$ {\it bridgeless}, if it contains no bridges.
The following lemma is easy to prove, by using that $e$ is a bridge if and only if there is
no circuit $C$ in ${\cal M}$ with $e\in C$.

\begin{lemma}
Let $\M$ be a matroid. Then
${\cal M}$ has a circuit decomposition if and only if ${\cal M}$ is 
bridgeless. Furthermore, if ${\cal M}$ is bridgeless, then every partial circuit decomposition of ${\cal M}$ can be extended to a circuit decomposition.
\end{lemma}

Note that every ear-decomposition is a circuit decomposition.
The following lemma extends a basic property of ear-decompositions to
circuit decompositions.

\begin{lemma} \label{lmm: lobe rank}
Let $(C_1,C_2,\dots,C_t)$ be a circuit decomposition of $\M$.
% Suppose $\mc C_1, \mc C_2, \ldots, \mc C_t$ is a sequence of circuits satisfying {E2} and {E3}. 
Then $$r(D_j) - r(D_{j - 1}) = |\tilde{C}_j| - 1$$
for all $1\leq j\leq t$.
\end{lemma}

\begin{proof}
Let $J=\tilde{C}_j$. The lemma is obvious for $|J|=1$, so we may assume that
$|J|\geq 2$. Let $e\in J$ and let $f\in J-e$. Now $f$ is a bridge in $D_{j}-e$, for otherwise
it is contained by a circuit $C'$ with $C' - D_{j - 1} \ne \emptyset$ and
$(C'-D_{j-1})\subsetneq (C_j-D_{j-1})$, contradicting (E2).
Hence $r(D_j)\geq r(D_{j - 1}) + |J| - 1$. As the inequality cannot be strict, the lemma follows.
\end{proof}

\section{Circuit decompositions of frameworks}
\label{sec:6}

In this section we consider cable-strut frameworks $(T,p)$.
Thus, unless stated otherwise, the underlying tensegrity
graph $T=(V,C\cup S)$ contains no bars and has no parallel members.

Let $(T,p)$ be a 
%generic\footnote{Define generic following Roth Whiteley} 
realization of the tensegrity graph $T=(V,C\cup S)$ in ${\R}^d$ and let
$F\subseteq C\cup S$ be a (possibly empty) subset of its members.
%$(T',p) be a subframework. 
We say that a dependence $\omega$ is a {\it semi-stress of $(T,p)$
with respect to} $F$ if $\supp \omega - F\not= \emptyset$ and
$\omega$ satisfies the sign constraints on $\supp \omega - F$.

Let $\omega$ be a semi-stress with respect to $F$ for which $\supp \omega - F$ is minimal,
and with respect to this, $\supp \omega \cap F$ is minimal.
We say that $\omega$ is a {\it minimal semi-stress} with respect to $F$.
Note that a (minimal) semi-stress with
respect to $F$
exists if $(T,p)$ has a proper stress and $F$ is a proper subset of $C\cup S$.

\begin{lemma}
\label{conn}
Let $(T,p)$ be a tensegrity framework with $T=(V,C\cup S)$, $F\subsetneq C\cup S$, and 
let $f\in (C\cup S)-F$.
Suppose that $(T,p)$ has a proper stress.
Then there exists a minimal semi-stress $\omega_f$ of $(T,p)$ with respect to $F$
with $f\in \supp \omega_f $.
\end{lemma}
 
\begin{proof}
Let $\omega$ be a semi-stress of $(T,p)$ with respect to $F$
with $f\in \supp \omega$, for which $\supp \omega-F$ is minimal, and with respect to this, $\supp \omega\cap F$ is minimal. Suppose that $\omega$ is not a minimal semi-stress of $(T,p)$ with respect to $F$. Then there exists a minimal semi-stress $\omega'$ of $(T,p)$ with respect to $F$ such that $f\notin \supp \omega'$ and $\supp \omega' -F \subsetneq \supp \omega -F$.
Let 
$$t=\min \left\{ \frac{\omega(e)}{\omega'(e)} : e\in \supp(\omega')-F\right\}.$$
Then $t>0$ and $\mu=\omega - t\omega'$ is a semi-stress of $(T,p)$ with respect to $F$, for which $f\in \supp \mu$
and $\supp \mu -F \subsetneq  \supp \omega -F$, contradicting the
minimality of $\omega$.
\end{proof}

\begin{lemma}
\label{dep}
Let $(T,p)$ be a generic $d$-dimensional realization of the tensegrity graph
$T=(V,C\cup S)$ and let $F\subsetneq C\cup S$.
Suppose that 
$\omega$ is a minimal semi-stress with respect to $F$. Then:\\
%$Z=\supp \omega$. Suppose that 
(a) if
$C$ is an ${\cal R}_d$-circuit in $\overline T$ with $C\subseteq F\cup \supp \omega$
and $C-F\not= \emptyset$, and
$\lambda_C$ is a dependence with $\supp \lambda_C = C$, then $\lambda_C$ or $-\lambda_C$ is a semi-stress with respect to $F$ such that
$\supp \lambda_C- F = \supp \omega - F$, \\
%Furthermore, for every $e\in (C\cup S)-F$ there exists a minimal semi-stress $\omega_e$ %with
%respect to $F$ with $e\in \supp \omega_e$.
(b) $\supp \omega$ is an ${\cal R}_d$-circuit.
\end{lemma}

\begin{proof}
Let $Z=\supp \omega$.
Since $(T,p)$ is generic, $\lambda_C$ is unique up to multiplication with non-zero scalars. 
By replacing $\lambda_C$ with $-\lambda_C$, if necessary,    
we may assume that the sign of $\lambda_C$
agrees with the sign of $\omega$ on at least one member in $C-F$.

If the sign pattern of $\lambda_C$ agrees with the sign pattern of $\omega$ on all members of $C-F$, then,
since $\omega$ is a semi-stress with respect to $F$, $\lambda_C$ is also a
semi-stress with respect to $F$. Moreover, the minimality of $\omega$ implies
$\supp \lambda_C- F = Z - F$.

So we may suppose that the two sign patterns do not agree everywhere in $C-F$. Let

$$t = \min \left\{ \frac{\omega(e)}{\lambda_C(e)} : e\in C-F,\, \sign(\lambda_C(e))= \sign(\omega(e))\right\}.$$

\noindent
Then $t>0$ and $\mu = \omega - t \lambda_C$ is a semi-stress with respect to $F$ for which
$\supp \mu - F$ is a non-empty proper subset of $Z-F$. 
This contradicts the minimality of $\omega$. Thus (a) follows.
We obtain (b) by applying (a) to an ${\cal R}_d$-circuit $C$ in $\overline T$, 
for which $C\subseteq Z$ and
$C-F\not= \emptyset$ hold.
\end{proof}

The main result of this section is as follows.
It shows that a generic, properly stressed tensegrity framework has
a ``geometric'' circuit decomposition.

\begin{theorem}
\label{geoear}
Let $(T,p)$ be a generic $d$-dimensional cable-strut framework with
$T=(V,C\cup S)$.
Suppose that $(T,p)$ has a proper stress and let $G=\overline T$. Then 
there exists a circuit decomposition $(C_1,C_2,\dots, C_t)$ of $G$
for which the subframework of $(T,p)$ on $D_j$ has a proper stress for all
$1\leq j\leq t$.
%Furthermore, if $T$ is connected, then there exists such a circuit decomposition
\end{theorem}

\begin{proof}
Since $(T,p)$ has a proper stress, ${\cal R}_d(G)$ is bridgeless. 
We shall prove the existence of the required circuit decomposition by
induction on the number of circuits.
Suppose that $(C_1,C_2,\dots, C_j)$ is a
partial circuit decomposition of $G$ with $D_j\not= E$ for which the subframework of $(T,p)$ on $D_i$ has a proper stress for all
$1\leq i\leq j$.
Let $\omega$ be a minimal semi-stress with
respect to $D_j$ with $\supp \omega = C$. 
By Lemma  \ref{dep} $C$ is an ${\cal R}_d$-circuit, which 
satisfies (E1) and (E2). Hence 
$(C_1,C_2,\dots, C_j,C)$ is a partial
circuit decomposition of $G$.
Suppose that $\omega'$ is a proper stress on $D_j$.
Then $\omega' + \varepsilon \omega$ 
is a proper stress on $D_j\cup C$, for 
a sufficiently small positive $\varepsilon$.
It follows that the circuit decomposition can be extended,
which completes the proof.
\end{proof}

We call the decomposition of Theorem \ref{geoear}  a {\it properly stressed circuit decomposition} of $(T,p)$.

We next show a key property of the properly stressed circuit decompositions
of minimally properly stressed frameworks.

\begin{theorem}
\label{min}
Suppose that $(T,p)$ is a minimally properly stressed generic cable-strut framework in ${\mathbb R}^d$ and let 
$(C_1,C_2,\dots, C_t)$ be a properly stressed circuit decomposition of $(T,p)$.
Then the subframework on
$D_j$ is minimally properly stressed for all $1\leq j\leq t$.
%in a properly
%stressed circuit decomposition induces a minimally properly stressed subframework.
\end{theorem}

\begin{proof}
%It suffices to prove the theorem for $j=t-1$.
We may assume that $t\geq 2$, since the subframework on $C_1$ is clearly minimally properly stressed. We may also assume that
$j=t-1$.
For a contradiction 
suppose that the subframework $(T',p')$ on $D_{t-1}-e$ is properly  stressed 
for some $e\in D_{t-1}$. If $e\notin C_t$, then
a proper stress of $(T',p')$ can be extended to a proper stress of $(T-e,p)$ (as in 
the proof of 
Theorem \ref{geoear}),
which contradicts the minimality of $(T,p)$. So we may assume that $e\in C_t$.
%By minimality e is in $T'\cap C$. 

Let $C'$ be a circuit in $D_{t-1}$ with $e\in C'$ and let $f\in C_t-D_{t-1}$.
By the strong circuit exchange axiom there is a circuit $C''$ with $f\in C''$,
$e\notin C''$, and $C''\subset C'\cup C_t$.
By (E2) we must have
%Theorem \ref{dep} we have $C_C''$ contains the lobe 
$C_t-D_{t-1}\subseteq C''$. It follows from Lemma \ref{dep}
that a dependence $\lambda$  of $C''$ is a semi-stress with respect to $D_{t-1}$.
Thus by adding an appropriate multiple of $\lambda$ to a proper stress of $(T',p')$ we obtain a proper
stress of $(T-e,p)$, contradicting minimality.
\end{proof}

\section{Upper bounds for generic cable-strut frameworks}
\label{sec:7}

In this section we prove the tight upper bound on the
size of a $d$-dimensional generic minimally infinitesimally rigid cable-strut framework,
for $1\leq d\leq 3$.
Since the proofs are very similar for these values of $d$, we prove the bound
for $d=3$ only. This way we can avoid a case analysis and the proof becomes more
transparent.
In fact we prove a more general result, which takes into account the 
degrees of freedom of the underlying graph.

\begin{theorem} \label{bound}
	Let $(T,p)$ be a $3$-dimensional generic cable-strut framework with $T = (V, C \cup S)$. Suppose that $(T,p)$ is minimally properly stressed.
 Then 
 $$|C \cup S| \le 4 |V| - 10 - 2 k_3 (T).$$
\end{theorem}

\begin{proof}
Let $(C_1,C_2, \ldots, C_t)$ be a properly stressed circuit decomposition of $(T,p)$. 
%It exists.
%a sequence of circuits given by lemma \ref{lmm: generalized ear decomposition}.
Let $T' = (V',E')$ be the subgraph of $T$ induced by the edges of $D_{t - 1}$,
let $E^+ = C_t - D_{t-1}$, $V^+ = V(C_t) - V(D_{t-1})$, and $V^- = V(C_t)\cap V(D_{t-1})$.
We may assume that $V^-\not= \emptyset$ by Lemma \ref{conn}.
%Let $\omega$ be a semi-stress of $T'$ with $\supp \omega = \mc C$. Our goal is to apply 
We prove the inequality by induction on $t$.
%induction on the size of $T$. 
For $t=1$, $E=C_1$ is a circuit, and the bound follows from
%Our base case is when $T = \mc C$ which is handled by 
%
%Lemma \ref{circuits}(3).
(\ref{ceq3}).

 By Lemma \ref{lmm: lobe rank}, we have 
 \begin{equation}
 \label{eq1}
 |E^+| = 3 |V^+| + 1 + k_3 (T') - k_3 (T).
 \end{equation}

 \noindent
 First suppose $|V^+|\geq 1$. The 4-edge-connectivity of $C_t$, and the fact that
 $V^-\not= \emptyset$ implies that $|E^+|\geq 2(|V^+|+1)$. Combining this with (\ref{eq1}) yields
 $1\leq |V^+|+k_3(T')-k_3(T)$. By plugging this into the RHS of (\ref{eq1}) we get
 $|E^+|\leq 4|V^+|+2k_3(T')-2k_3(T)$.
The induction hypothesis now gives
$$|E|=|E(T')|+|E^+|\leq 4|V(T')|-10-2k_3(T')+4|V^+|+2k_3(T')-2k_3(T)=4|V|-10-2k_3(T),$$
as required.

 \iffalse
 This equality 
 
 Since $T'$ is a disjoint union of its connected components, we have $k_d (T') = (q - 1) \binom{d + 1}2 + \sum_{i = 1}^q k_d (T_i)$. Applying the induction hypothesis to each $T_i$ and plugging these formulas in, we get \(
		\label{eq: edge upper bound}
		\begin{aligned}
			|C \cup S| &=
			\sum_{i = 1}^q |E_i| + \frac{d + 1}{d - 1} |E^+| - \frac2{d - 1} |E^+| \\ &\le
			(d + 1) |V| - \binom{d + 2}2 - \frac{d + 1}{d - 1} k_d (T) + \frac{d + 1}{d - 1} \qty(|V^+| + q - \frac{2 |E^+|}{d + 1}) .
		\end{aligned} \)
	In order for $T$ to have a proper stress, it must be $(d + 1)$-edge connected. If $|V^+| + q \ge 2$, then we have $|E^+| \ge \frac12 (d + 1) (|V^+| + q)$, and plugging this into \eqref{eq: edge upper bound} yields the desired inequality.
	
	Otherwise, $T'$ is connected and $V^+ = \emptyset$. If $|E^+| \ge \frac12 (d + 1)$ is satisfied, then we still get the desired inequality as in the previous paragraph. Since $T$ is minimal, we must have $|E^+| \ge 2$. This completes the proof for $d \le 3$. For $d = 4$ and 5, the only problem is if $|E^+| = 2$ in which case $k_d (T') - k_d (T) = 1$. But then we have \[
		|E'| \le (d + 1) |V'| - \binom{d + 2}2 - \frac{d + 1}{d - 1} k_d (T') \qc
		|C \cup S| \le (d + 1) |V| - \binom{d + 2}2 - \frac{d + 1}{d - 1} k_d (T) + \frac{d - 3}{d - 1} . \]
	Analyzing the values of $k_d (T')$ and $k_d (T)$ mod 3 or 2 (for $d = 4$ and 5, respectively), we again get the desired inequality.
 \fi

 Next suppose $V^+=\emptyset$.
By minimality, we have $|E^+|\geq 2$, which gives $k_3(T')>k_3(T)$ 
and $|E^+|\leq 2k_3(T')-2k_3(T)$
by \eqref{eq1}. Hence
$$|E|=|E(T')|+|E^+|\leq 4|V(T)|-10-2k_3(T')+2k_3(T')-2k_3(T)=4|V|-10-2k_3(T).$$
This completes the proof.
\end{proof}

By Lemma \ref{minps} and Theorem \ref{bound}
it follows that a $3$-dimensional minimally infinitesimally rigid generic  cable-strut  framework with $T = (V, C \cup S)$ has at most $4|V|-10$ members.
By slightly modifying the counts in the above proof we can also determine the tight
upper bounds in the cases $d=1$ and $d=2$ in a similar manner.
We summarize these results in the following statement.

\begin{figure}[!h]
\begin{center}
\includegraphics[width=5cm]{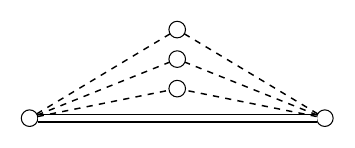}
\caption{\label{2v3} A tensegrity graph with $2|V|-3$ members that has a minimally infinitesimally rigid realization in $\R ^1$.
%A minimally rigid cable-strut framework in $\mathbb R^1$ with $2|V|-3$ members.
}
\end{center}
%\label{1dim}
\end{figure}

\begin{figure}[!h]
\begin{center}
\includegraphics[width=5cm]{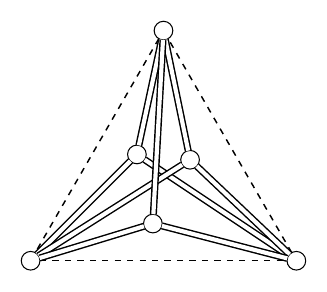}
\includegraphics[width=5cm]{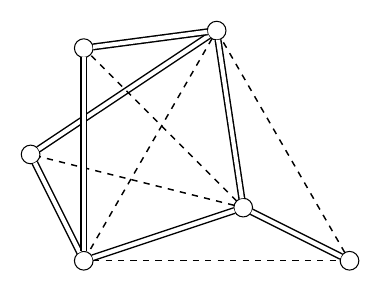}
\caption{\label{3v6} Minimally infinitesimally rigid tensegrity frameworks in the plane with $3|V|-6$ members.}
\end{center}
\end{figure}

\begin{theorem} \label{minrigbound}
	Let $(T,p)$ be a $d$-dimensional generic minimally infinitesimally rigid cable-strut framework with $T = (V, C \cup S)$, where $1\leq d\leq 3$. 
 Then 
 \begin{equation}
 \label{4v}
|C \cup S| \le (d+1) |V| - \binom{d + 2}2.
%4 |V| - 10.
\end{equation}
\end{theorem}

%The 1-dimensional result implies Theorem from \cite{CJP}.

The upper bound is best possible, see Figures \ref{2v3}, \ref{3v6}, and \ref{4v10},
for $d=1$, $d=2$, and $d=3$, respectively.
%For more tight examples see Figure \ref{3v6}.
We believe that the bound in (\ref{4v}) is the right answer for all $d\geq 1$.

\begin{figure}[!h]
\begin{center}
\includegraphics[width=5cm]{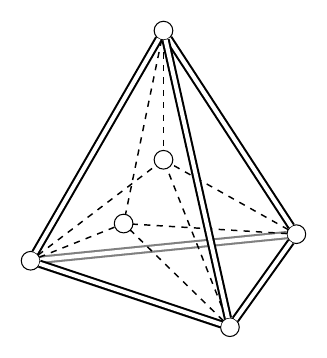}
\caption{\label{4v10} A minimally infinitesimally rigid tensegrity framework in ${\mathbb R}^3$ with
$4|V|-10$ members.}
\end{center}
\end{figure}

\begin{conjecture}
\label{conj_cs}
Let $(T,p)$ be a $d$-dimensional generic minimally infinitesimally rigid cable-strut framework with $T = (V, C \cup S)$, for some positive integer $d$. 
%Let $(T,p)$ be a $d$-dimensional generic realization of a tensegrity graph $T = (V, C \cup %S)$ with no parallel members. Suppose that $(T,p)$ is minimally infinitesimally rigid.
 Then 
 \begin{equation}
 \label{dv}
|C \cup S| \le (d+1) |V| - \binom{d + 2}2.
\end{equation}
\end{conjecture}

\subsection{Higher dimensions}

For $d\geq 4$ 
%we have not found the tight upper bound yet, when parallel members are
%not allowed, but 
we can improve on the general upper bound of Theorem \ref{th_parallel} and get closer to the conjectured bound in (\ref{dv}).
To obtain this improvement, we need the following lemma, which is 
an extension
%\footnote{In what sense? Can we unify these lemmas?} 
of Lemma \ref{circuits}.
%from the March 29 draft of the minimally rigid tensegrity frameworks paper. 
We say that a 
%${\cal R}_d$-
circuit decomposition $(C_1,C_2,\dots,C_t)$ of ${\cal R}_d(G)$
is {\it strict}, if every circuit adds at least one new vertex, that is,
$V(C_j)-V(D_{j-1})\not= \emptyset$ for all $1\leq j\leq t$.

It will be convenient to work with the rank $r^*_d(G)$ of the dual of the rigidity matroid of a graph $G$, given by $r_d^* (G) =|E| - r_d (G)$. 
%If $G$ has at least $d+2$ vertices, then we have
It is useful to observe that for every circuit decomposition $(C_1,C_2,\dots,C_t)$
of ${\cal R}_d(G)$ we have $t=r_d^*(G)$.

\begin{lemma} \label{lmm: strict ear decomposition}
	Let $G = (V,E)$ be a connected graph. If ${\cal R}_d(G)$ has a strict circuit decomposition, then the following are true:
%	\begin{enumerate}
\begin{equation}
\label{L8:1}
  |E| \ge \frac12 (d + 1) (|V| + r_d^* (G) - 1)
  \end{equation}
  \begin{equation}
  \label{L8:2}
		k_d (G) \le \frac12 (d - 1) (|V| - (r_d^* (G) + d + 1))
   \end{equation}
  \begin{equation}
  \label{L8:3}
		r_d^* (G) \le |V| - \frac2{d - 1} k_d (G) - (d + 1)
	\end{equation}
\end{lemma}
\begin{proof}
	Note that (\ref{L8:2}) and (\ref{L8:3}) are equivalent. Since $G$ contains an
 ${\cal R}_d$-circuit, we must have $|V|\geq d+2$. Hence
  (\ref{L8:2}) follows from (\ref{L8:1}) and the formula 
 \begin{equation}
 \label{L8:4}
% \[
		|E| =
		r_d (G) + r_d^* (G) =
		d |V| - \binom{d + 1}2 - k_d (G) + r_d^* (G) . 
  %\]
\end{equation}	
 Thus, it suffices to show (\ref{L8:1}). Let $(C_1, \ldots, C_t)$ be a strict
 circuit decomposition of ${\cal R}_d(G)$. Note that
% $\mc R_d$-circuit decomposition of $G$ where 
$t = r_d^* (G)$. We prove the lemma by induction on $t$.

For $t=1$ the statement follows from
the facts that an $\mc R_d$-circuit $G$ is $(d + 1)$-edge connected and has $r_d^*(G)=1$. 
%(\ref{ceq2}) and the fact that $r_d^*(G)=1$, if $G$ is an ${\cal R}_d$-circuit.
So we may assume that $t\geq 2$ holds.

%implies
%is obvious for $t = 1$ (which is the content of Lemma 4) since the minimum degree of %an $\mc R_d$-circuit is $d + 1$. We'll proceed by induction on $t$.
	
	Let $H_1, \ldots, H_q$ be the connected components of $D_{t - 1} = C_1 \cup \cdots \cup C_{t - 1}$, for some $q\geq 1$. Then the vertex sets $V(H_i)$ are pairwise disjoint and each $V(H_i)$ intersects $V(C_t)$ by the connectivity of $G$. Let $V^+ = V(C_t) - \bigcup_{i = 1}^q V(H_i)$ and 
 $E^+ = \tilde C_t$.
 %- \bigcup_{i = 1}^q H_i$. 
 Then $|E^+| \ge \frac12 (d + 1) (|V^+| + q)$, since $C_t$ is $(d + 1)$-edge connected. 
 Observe that the existence of a strict circuit decomposition of ${\cal R}_d(G)$
implies that each ${\cal R}_d(H_i)$ has a strict circuit decomposition. 
Thus we may apply
 induction to the $H_i$'s and deduce that \[
		|E| =
		|E^+| + \sum_{i = 1}^q |H_i| \ge
		\frac12 (d + 1) (|V| + r_d^* (G) - 1) \]
	which completes the proof.
\end{proof}

The graphs obtained from the complete bipartite graph
$K_{d+1,q}$, for some $q\geq 1$, by adding a complete graph on the
color class of size $d+1$ achieve equality in (\ref{L8:1}), (\ref{L8:2}), and  (\ref{L8:3}).

%The graphs which achieve equality\footnote{Where do we need the "strict" property in the proof? Which inequality do we mean here?} are the connected graphs whose blocks are a union of some number of copies of $K_{d + 2}$ along $d + 1$ common vertices. 

We next use Lemma \ref{lmm: strict ear decomposition} to 
obtain better bounds for $d\geq 4$ as well as a 
different proof of Theorem \ref{minrigbound}.

\begin{theorem} \label{thm: edge bounds}
	Let $(T,p)$ be a $d$-dimensional generic minimally infinitesimally rigid cable-strut framework with $T=(V,C\cup S)$. Then 
 \begin{equation}
% \(
		\label{eq: general d edge bound}
		|E| \le \frac12 (3d - 1) |V| - \binom{d + 2}2 - \frac12 (d - 3) (d + 2) 
  %\)
\end{equation}	
 for $d \ge 3$, and
 %\(
	\begin{equation}
  \label{eq: d at most 3 edge bound}
		|E| \le (d + 1) |V| - \binom{d + 2}2 
  %\)
\end{equation}	
 for $d \le 3$.
\end{theorem}

\begin{proof}
It follows from Theorem \ref{geoear} and Lemma \ref{conn} that
%	By using Lemma \ref{lmm: choose edge} and some results from previous sections 
 %on minimally properly stressed frameworks from the draft, 
 %we can show that 
 $(T,p)$ has
 a properly stressed circuit decomposition $(C_1, \ldots, C_{t_1}, C_{t_1 + 1}, \ldots, C_{t_2})$ such that the first $t_1$ lobes add vertices and the last $t_2 - t_1$ lobes add no vertices. Let $T' = C_1 \cup \cdots \cup C_{t_1}$. It follows from the minimality of
 $(T,p)$ 
 %equation (7) of the draft\footnote{check number} 
 that $|C \cup S| \le |E(T')| + 2 k_d (T')$. Since $|E(T')| = d |V| - \binom{d + 1}2 - k_d (T') + r_d^* (T')$, 
inequality (\ref{L8:2}) implies \[
		|C \cup S| \le
		\frac12 (3d - 1) |V| - \binom{d + 2}2 - \frac12 (d - 3) (r_d^* (T') + d + 1) . \]
	Plugging in $r_d^* (T') \ge 1$ gives us (\ref{eq: general d edge bound}), and plugging in $r_d^* (T') + d + 1 \le |V|$ gives us (\ref{eq: d at most 3 edge bound}).
\end{proof}

The main results of this section are concerned with generic $d$-dimensional infinitesimally rigid cable-strut frameworks $(T,p)$. 
%These frameworks contain
%no bars and no parallel members. Furthermore,
The underlying graph $\overline{T}$ of such a framework is 
{\it redundantly rigid}, that is, it remains rigid in $\R^d$ after the
removal of any member of $T$.
In these results the assumption on the non-existence of parallel members can be 
replaced by the weaker assumption, saying that 
%omitted:
%each of these results remains valid 
%for all tensegrity frameworks $(T,p)$ that contain no bars and satisfy that
$\overline{T}$ is 
redundantly rigid. 
This follows from the fact that in this case
$(T,p)$ contains a minimally infinitesimally rigid spanning cable-strut 
framework. We omit the details.

\section{Minimally ${\cal R}_d$-connected graphs}
\label{sec:8}

Our upper bounds on the size of minimally infinitesimally rigid tensegrity frameworks in $\R^d$
imply similar upper bounds on the size of certain graphs defined by specific rigidity 
properties. In this section we collect some corollaries of this type, including an
affirmative answer to a recent conjecture of the second author.

An ${\cal R}_d$-connected graph $G$ is called {\it minimally ${\cal R}_d$-connected},
if $G-e$ is not ${\cal R}_d$-connected for all $e\in E(G)$.
The cases $d=1,2$ of the following conjecture were settled in \cite{Jear}.

\begin{conjecture} \cite{Jear}
\label{conjecture:minimallyMconnected}
Let $G=(V,E)$ be a minimally ${\cal R}_d$-connected graph.
Then we have
\begin{equation}
|E|\leq (d+1)|V|-\binom{d+2}{2},
\end{equation}
where equality holds if and only if $G=K_{d+2}$. 
%Furthermore,
%if $|V|\geq \binom{d+1}{2}+d+2$, then
%\begin{equation}
%|E|\leq (d+1)|V|-(d+1)^2.
%\end{equation}
\end{conjecture}

%A sharper version 

%The cases $d=1,2$ were settled in \cite{Jear}.
By using our results on circuit decompositions of matroids, we
can verify the $d=3$ case of Conjecture \ref{conjecture:minimallyMconnected}
in a stronger form.
We need the following lemma.

\begin{lemma} \label{minear}
\cite{Jear}
Let $G=(V,E)$ be a minimally ${\cal R}_d$-connected graph and let
$(C_1, C_2,\ldots, C_{t})$  be an ear-decomposition.
Then $D_j$ induces a  minimally ${\cal R}_d$-connected subgraph for all $1\leq j\leq t$.
\end{lemma}

The main result of this section is as follows.

\begin{theorem}
\label{Mconn}
Let $G=(V,E)$ be a minimally ${\cal R}_d$-connected graph for some $d\geq 3$. Then we have
\begin{equation}
|E| \le \frac12 (3d - 1) |V| - \binom{d + 2}2 - \frac12 (d - 3) (d + 2)-2k_d(G)
\end{equation}
%\end{theorem}
with equality if and only if $G=K_{d+2}$.
\end{theorem}

\begin{proof}
Let $(C_1,C_2,\dots,C_{t_1},C_{t_1+1},\dots, C_t)$ be an ear-decomposition of ${\cal R}_d(G)$.
We may assume\footnote{Here we use the following statement, which can be proved by using the strong circuit axiom: given a partial ear-decomposition $(C_1,C_2,\dots,C_{j})$ and an element $e\in E-D_j$
of a connected matroid, there exists a
partial ear-decomposition $(C_1,C_2,\dots,C_{j},C_{j+1})$ with $e\in C_{j+1}$.}
that the lobes of the first $t_1$ ears add new vertices and the
remaining $t-t_1$ lobes add no vertices. Let $H$ be the spanning subgraph of $G$ induced
by $D_{t_1}$.
It follows from Lemma \ref{minear} that $|\tilde C_i|\geq 2$,
and hence the addition of $C_i$ increases the rank by at least one,
for all $t_1+1\leq i\leq t$. 
%Let $H$ be the spanning subgraph of $G$ induced
%by $D_{t_1}$. Then 
We can use this fact, together with (\ref{L8:4}), (\ref{L8:2}), and the inequality $r_d^*(H)\geq 1$
to deduce that
\begin{align*}
|E| &\leq |E(H)| + 2(k_d(H)-k_d(G)) = d|V| - \binom{d + 1}2 +k_d(H)+r^*_d(H)-2k_d(G)\\
&\leq \frac12 (3d - 1) |V| - \binom{d + 1}2 - \frac12 (d - 1) (r_d^*(H) + d+1)+ r_d^*(H) -2k_d(G)\\
&=\frac12 (3d - 1) |V| - \binom{d + 2}2 - \frac12 (d - 3) (r_d^*(H) + d + 1)-2k_d(G)\\
&\leq \frac12 (3d - 1) |V| - \binom{d + 2}2 - \frac12 (d - 3) (d + 2)-2k_d(G).
\end{align*}

Equality holds only if $H$ is an ${\cal R}_d$-circuit that achieves equality in (\ref{L8:2}). 
In that case $H$, and hence also $G$, must be a $K_{d+2}$.
%Since $H$ is minimally
%${\cal R}_d$-connected by Lemma \ref{minear}, this would imply $H=K_{d+2}$ and hence
%$G=K_{d+2}$.
\end{proof}

Theorem \ref{Mconn}
improves on the previously known upper bound of \cite{Jear} and
implies an affirmative answer to the $d=3$ case
%\footnote{Does it imply an improvement for all $d\geq 3$? Say, $\frac{3}{2}d$?} 
of
Conjecture \ref{conjecture:minimallyMconnected}.

\begin{corollary}
Let $G=(V,E)$ be a minimally ${\cal R}_3$-connected graph. Then we have
\begin{equation}
|E| \le 4|V| - 10,
\end{equation}
%\end{theorem}
with equality if and only if $G=K_{5}$.
\end{corollary}

\subsection{Further corollaries}

A tensegrity graph $T=(V,B\cup C\cup S)$ is called {\it weakly rigid} in $\R^d$ if there exists
an infinitesimally rigid generic realization $(T,p)$ of $T$ in $\R^d$.
It is said to be {\it minimally weakly rigid} in $\R^d$ if $T-e$ is not weakly rigid in $\R^d$ for all
$e\in B\cup C\cup S$.
Since every infinitesimally rigid generic realization of a minimally weakly rigid tensegrity graph
is minimally infinitesimally rigid, 
our upper bounds on the size of minimally infinitesimally rigid frameworks imply similar upper bounds on the size of minimally weakly rigid tensegrity graphs.

In particular, we can use Theorem \ref{minrigbound} to deduce the tight upper bounds on the
size of minimally weakly rigid cable-strut graphs in $\R^d$ for $1\leq d\leq 3$.
In the special case $d=1$ this bound was already obtained in \cite{CJP}.

The underlying tensegrity graphs of the frameworks in Figure \ref{2v3}, \ref{3v6},
and \ref{4v10}
show that these upper bounds are tight.
We remark that there exist generic minimally infinitesimally rigid frameworks in $\R^d$ whose underlying tensegrity graphs are not minimally weakly rigid in $\R^d$.

%\subsection{Minimally redundantly rigid graphs}

\medskip

%A rigid graph $G=(V,E)$ in ${\mathbb R}^d$ is said to be {\it redundantly rigid}
%if $G-e$ is rigid for all $e\in E$. It
A redundantly rigid graph $G$ in $\R^d$ is called {\it minimally redundantly rigid}
if $G-e$ is not redundantly rigid in $\R^d$ for all $e\in E(G)$.
Let $G=(V,E)$ be a redundantly rigid graph in $\R^d$.
It is not hard to see that a $d$-dimensional
generic realization $(G,p)$ of $G$ 
has a nowhere zero dependence $\omega$, see e.g. \cite{JRS}.
Let $T=(V,C\cup S)$ be a tensegrity graph with $C\cup S=E$ in which $e$
is a cable if and only if $\omega(e)<0$.
%Define members by using the sign pattern.
Then $(T,p)$ is infinitesimally rigid by Theorem \ref{fund}.
Moreover, if $G$ is minimally redundantly rigid, then $(T,p)$ is
minimally infinitesimally rigid in $\R^d$.
Thus we can use 
Theorem \ref{thm: edge bounds} 
to deduce upper bounds on the size of minimally redundantly rigid graphs  
in $\R^d$.
See \cite{Jear,K} for related results.

%A recent paper
%of Kir\'aly \cite{K} provides almost tight upper bounds for the
%size of a minimally redundantly rigid graph in ${\mathbb R}^d$ for all $d\geq 1$.

\section{Concluding remarks}

In this paper we have determined, among others, the best possible upper bound on the number of members of a $d$-dimensional minimally infinitesimally rigid
tensegrity framework (resp. generic cable-strut framework) on $n$ vertices, in terms of $d$ and $n$, for all $d\geq 1$ (resp. $1\leq d \leq 3$).
The following related questions are left open.  
%\noindent

%(i) What is the best possible upper bound in the case of
%cable-strut frameworks? 

\medskip
\noindent
(i) What is the best possible upper bound in the case of (non-generic)
cable-strut frameworks? 
For example, 
consider a tensegrity framework $(T,p)$ in $\R^2$
obtained from a unit square of cables
by adding some vertices, which are mapped to the center of the square, 
and are connected to the vertices of the square by struts. See Figure \ref{nongen} (left).
These frameworks are minimally infinitesimally rigid, and their
tensegrity graphs $T=(V,C\cup S)$ satisfy 
$|C\cup S|=4|V|-12$.
This count is quite close to the bound of Theorem~\ref{th_parallel}, but there is
a gap. W. Whiteley \cite{Whiteley} conjectures that for $d=2$ the tight
upper bound is $4|V|-12$.
\medskip

\begin{figure}[!h]
\begin{center}
\includegraphics[width=5cm]{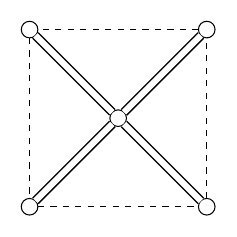}
\includegraphics[width=5cm]{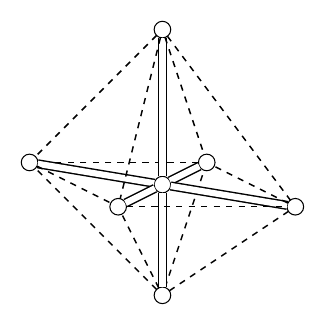}
\caption{\label{nongen} Nongeneric minimally infinitesimally rigid cable-strut frameworks in $\R^2$ and $\R^3$.}
\end{center}
\end{figure}

%\noindent (ii) We have identified the extremal tensegrity frameworks in the
%case of Theorem \ref{th_parallel}, but 
%the corresponding questions remain open with respect to
%Theorem \ref{th_parallel_bars} and Theorem \ref{minrigbound}.
%\medskip

\noindent (ii) 
In a minimally infinitesimally rigid bar-and-joint framework the
sparsity count extends to all subsets of the vertex set.
Is there a similar phenomenon in the case of 
minimally infinitesimally rigid tensegrity frameworks?

%Do the sparsity bounds extend to all subsets of the
%vertex set, like in the case of bar-and-joint frameworks?
%For example, is it true that for each vertex set $X$ of size at least $d$, $X$ induces at most $2d|X|-2{d+1\choose 2}$ members in  a minimally infinitesimally rigid
%cable-strut framework?

\section{Acknowledgements}

%This work was supported by the Hungarian Scientific Research Fund provided by the National Research, Development and Innovation Office, grant No. K135421. 

The second author was supported by the MTA-ELTE Momentum Matroid Optimization Research Group and the National Research, Development and Innovation Fund of Hungary, financed under the ELTE TKP 2021‐NKTA‐62 funding scheme.


\begin{thebibliography}{99}

 %\bibitem{RW}
%	{\sc B. Roth and W. Whiteley}, Tensegrity frameworks, {\it Trans. Amer. Math. Soc.}, 265 (1981), 419-446.
	
%	\bibitem{JJconnrig} {\scshape B. Jackson and T. Jord\'an,}
%	\newblock Connected rigidity matroids and unique
%	realizations of graphs, 
%	{\it J. Combin. Theory Ser. B}, Vol. 94, 1-29, 2005.

%\bibitem{Steinitz} E. Steinitz. Bedingt konvergente Reihen und konvexe Systeme. i-ii-iii. J. Reine Angew. Math., 143 (1913), 128–175, 144 (1914), 1–40, 146 (1916), 1-52.

\bibitem{BonniceReay} {\scshape W. E. Bonnice and J. R. Reay}, Interior points of convex hulls,
{\it Israel Journal of Mathematics}, 1966, 4: 243-248.

%\bibitem{Jor} T. Jordán. Combinatorial rigidity: graphs and matroids in the theory of rigid frameworks. Discrete Geometric Analysis, MSJ Memoirs, vol. 34, pp. 33-112, 2016.

%\bibitem{Soltan} V. Soltan. Lectures on convex sets. World Scientific, 2019.

\bibitem{CJP} {\scshape A. Clay, T. Jord\'an, and J. Palmer},
Minimally weakly globally rigid graphs on the line, manuscript.

\bibitem{CH} {\scshape C.R. Coullard and L. Hellerstein},
Independence and port oracles for matroids, with an application to computational learning theory, {\it Combinatorica}  Vol. 16, pp. 189–208, (1996)


\bibitem{con} {\scshape R. Connelly},
The rigidity of certain cabled frameworks and the second-order rigidity of arbitrarily triangulated convex surfaces, {\it Advances in Mathematics} 37, 272-299, 1980.

\bibitem{CG} {\scshape R. Connelly and S.D. Guest}, Frameworks, tensegrities, and
symmetry, Cambridge University Press, 2022.
\newblock


\bibitem{CN} {\scshape R. Connelly and A. Nixon},
\newblock
Tensegrity, Chapter 10 in: Handbook of Geometric Constraint Systems Principles,
eds: M. Sitharam, A. St. John, and J. Sidman, CRC Press, 2019.

\bibitem{gluck}
{\scshape H.~Gluck},
\newblock Almost all simply connected closed surfaces are rigid.
\newblock In {\em Geometric Topology}, volume 438 of {\em Lecture Notes in
  Mathematics}, pages 225--239. Springer-Verlag, 1975.
  
%\bibitem{GGJN} {\sc G. Grasegger, H. Guler, B. Jackson, and A. Nixon},
%	Flexible circuits in the $d$-dimensional rigidity matroid,
%	{\it J. Graph Theory}, Vol. 100, 315–330, 2022.
 
 
 \bibitem{JJsparse} {\scshape B. Jackson and T. Jord\'an,}
	\newblock The $d$-dimensional rigidity matroid of sparse graphs,
	{\it J. Combin. Theory Ser. B}, Vol. 95, 118-133, 2005.
 
 \bibitem{JJdress} {\scshape B. Jackson and T. Jord\'an,} The Dress conjectures on
rank in the $3$-dimensional rigidity matroid, 
{\it Advances in Applied Mathematics}, Vol. 35, 
%Issue 4.
355-367, 2005.

\bibitem{JJ} {\sc B. Jackson and T. Jord\'an},
	On the rank function of the 3-dimensional rigidity matroid
	{\it International Journal on Computational Geometry and Applications}, Vol. 16, Nos. 5-6 (2006) 415-429.

\bibitem{JJK} {\scshape B. Jackson, T. Jord\'an, and Cs. Kir\'aly},
 Strongly rigid tensegrity graphs on the line,
{\it Discrete Applied Mathematics} 161, 1147-1149, 2013.

\bibitem{Jor} {\scshape T. Jordán}, Combinatorial rigidity: graphs and matroids in the theory of rigid frameworks. Discrete Geometric Analysis, {\it MSJ Memoirs}, vol. 34, pp. 33-112, 2016.

\bibitem{Jear} {\sc T. Jord\'an,} Ear-decompositions, minimally connected matroids, and rigid graphs, 
{\it J. Graph Theory}, Vol. 105, Issue 3, March 2024, pp. 451-467.

 
 
 
 \bibitem{JRS} {\sc T. Jord\'an, A. Recski, and Z. Szabadka,}
Rigid tensegrity labelings of graphs, {\it European J. Combinatorics} 30 (2009) 
1887-1895.

%\bibitem{K} {\sc Cs. Kir\'aly}
\bibitem{K}  {\scshape Cs. Kir\'aly}, On the size of highly redundantly rigid graphs,
Proc. 12th Japanese-Hungarian Symposium on Discrete Mathematics and Its Applications, Budapest, March 2023,
pp. 263-272.

\bibitem{Oxley} {\scshape J. Oxley}, Matroid Theory, Oxford University Press, 1992.

\bibitem{RW}
	{\sc B. Roth and W. Whiteley}, Tensegrity frameworks, {\it Trans. Amer. Math. Soc.}, 265 (1981), 419-446.

\bibitem{SW} {\sc B. Schulze and W. Whiteley},
Rigidity and scene analysis, {\it in} 
%Handbook of Discrete and Computational Geometry,
%3rd edition, 
J.E. Goodman, J. O'Rourke, C.D. T\'oth (eds.),
 {\it Handbook of Discrete and Computational
Geometry,} 3rd ed., CRC Press, Boca Raton, 
%pp. 1661-1694, 
2018.
%CRC Press, 2018.

\bibitem{Soltan} {\scshape V. Soltan}, Lectures on convex sets. World Scientific, 2019.

\bibitem{Steinitz} {\scshape E. Steinitz}, Bedingt konvergente Reihen und konvexe Systeme. i-ii-iii. {\it J. Reine Angew. Math.}, 143 (1913), 128–175, 144 (1914), 1–40, 146 (1916), 1-52.

\bibitem{Whiteley} {\scshape W. Whiteley}, Tensegrity, draft chapter, 1987.


\end{thebibliography}
\end{document}